\documentclass{amsart}

\usepackage{latexsym}
\usepackage{epsfig}
\usepackage{amsthm}
\usepackage{amssymb}
\usepackage{amsmath}
\usepackage{amscd}
\usepackage{xcolor}
\usepackage[english,main=french]{babel}
\usepackage{graphicx}
\usepackage{caption}
\usepackage{subcaption}
\usepackage{float}

\usepackage{cleveref}
\Crefname{appendix}{Annexe}{Annexes}

\newtheorem{thm}{Théorème}[section]
\newtheorem{defn}[thm]{Définition}
\newtheorem{prop}[thm]{Proposition}

\newtheorem{lem}[thm]{Lemme} 

\newtheorem{conj}[thm]{Conjecture}

\theoremstyle{remark}

\expandafter\chardef\csname pre amssym.def at\endcsname=\the\catcode`\@
\catcode`\@=11
\def\undefine#1{\let#1\undefined}
\def\newsymbol#1#2#3#4#5{\let\next@\relax
 \ifnum#2=\@ne\let\next@\msafam@\else
 \ifnum#2=\tw@\let\next@\msbfam@\fi\fi
 \mathchardef#1="#3\next@#4#5}
\def\mathhexbox@#1#2#3{\relax
 \ifmmode\mathpalette{}{\m@th\mathchar"#1#2#3}%
 \else\leavevmode\hbox{$\m@th\mathchar"#1#2#3$}\fi}
\def\hexnumber@#1{\ifcase#1 0\or 1\or 2\or 3\or 4\or 5\or 6\or 7\or 8\or
 9\or A\or B\or C\or D\or E\or F\fi}

\font\teneufm=eufm10
\font\seveneufm=eufm7
\font\fiveeufm=eufm5
\newfam\eufmfam
\textfont\eufmfam=\teneufm
\scriptfont\eufmfam=\seveneufm
\scriptscriptfont\eufmfam=\fiveeufm

\catcode`\@=\csname pre amssym.def at\endcsname

\newcommand{\N}{{\bf N}}

\newcommand{\Z}{{\bf Z}}

\usepackage{amssymb}
\usepackage[all]{xy}
\usepackage{enumerate}

\def    \Z  {{\mathbb Z}}
\def    \N  {{\mathbb N}}

\def\qedbox{$\square$}%
\def\qed{\ifmmode\qedbox\else\unskip\ \hglue0mm\hfill
     \qedbox\smallskip\goodbreak\fi}%
\usepackage{xcolor}
\begin{document}

\title[Graphes dans les surfaces et ergodicité topologique]
{Graphes dans les surfaces et ergodicité topologique}

\author[Dustin Connery-Grigg, Fran\c{c}ois Lalonde et Jordan Payette]{Dustin Connery-Grigg, Fran\c{c}ois Lalonde et Jordan Payette}

\address{Connery-Grigg: Département de mathématiques et de statistique, Université de Montréal, Montréal, Québec, Canada et Institut de Mathématiques de Jussieu-Paris Rive Gauche, Sorbonne Université, France. Lalonde: Département de mathématiques et de statistique, Université de Montréal, Montréal, Québec, Canada et Stanford University, CA, USA. Payette : Department of Mathematics and Statistics, McGill University, Montréal, Québec, Canada}
\email{connery@imj-prg.fr; lalonde@dms.umontreal.ca;jordan.payette@mail.mcgill.ca}

\thanks{Le premier auteur est soutenu par une bourse doctorale du Fonds de recherche du Québec\,--\,Nature et technologie (FRQNT) et par une bourse postdoctorale de la Fondation des sciences mathématiques de Paris (FSMP). Le second est soutenu par le programme des chaires de recherches du Canada en géométrie et topologie symplectiques et par une subvention du Conseil de recherches en sciences naturelles et en génie du Canada (CRSNG). Le troisième est soutenu par une bourse postdoctorale du FRQNT}

\date{}

      \bigskip \bigskip

\maketitle

\bigskip
\bigskip \bigskip \bigskip \bigskip \bigskip \bigskip \bigskip \bigskip \bigskip \bigskip \bigskip \bigskip \bigskip \bigskip \bigskip \bigskip \bigskip 
\noindent
\begin{center}
\end{center}

  \bigskip    
\begin{otherlanguage}{english} 
 \begin{abstract} The simplest way to make a dynamical system out of a finite connected graph $G$ is to give it a polarization, that is to say a cyclic ordering of the edges incident to a vertex, for each vertex. The phase space $\mathcal{P}(G)$ then consists of all pairs $(v,e)$ where $v$ is a vertex and $e$ is an edge incident to $v$. Such an initial condition gives a position and a momentum.  The data $(v,e)$ is of course equivalent to an edge endowed with an orientation $e_{\mathcal O}$. With the polarization, each initial data leads to a leftward walk defined by turning left at each vertex, or making a rebound if there is no other edge. A leftward walk is called complete if it goes through all edges of $G$, not necessarily in both directions. As usual, we define the valence of a vertex as the number of edges incident to it, and we define the valence of a graph as the average of the valences of its vertices. In this article, we prove that if a graph which is embedded in a closed oriented surface of genus $g$ admits a complete leftward walk, then its valence is at most $1 + \sqrt{6g+1}$. We prove furthermore that this result is sharp for infinitely many genera $g$, and that it is asymptotically optimal as $g \to + \infty$. This leads to obstructions for the embeddability of graphs on a surface in a way which admits a complete leftward walk. Since checking that a polarized graph admits a complete leftward walk or not is done in time $4N$, where $N$ is the cardinality of the edges (we just have to check it on both orientations of any given edge), this obstruction is particularly efficient in terms of computability.  This problem has its origins in interesting consequences for what we will call here the {\it topological ergodicity} of conservative systems, especially Hamiltonian systems $H$ in two dimensions where the existence of a complete leftward walk corresponds to a topologically ergodic orbit of the system, i.e. an orbit of $H$ visiting all the topology of the surface. We limit ourselves here to two dimensions, but generalisations of this theory should hold for autonomous Hamiltonian systems on a symplectic manifold of any dimension.

\bigskip
\noindent 
R\begin{tiny}ÉSUMÉ\end{tiny}. La fa\c{c}on la plus simple de faire d'un graphe fini connexe $G$ un système dynamique est de lui donner une polarisation, c'est-à-dire un ordre cyclique des arêtes incidentes à chaque sommet. L'espace de phase $\mathcal{P}(G)$ d'un graphe consiste en toutes les paires $(v,e)$ où $v$ est un sommet et $e$ une arête incidente à $v$. Elle donne donc la position et le vecteur initiaux.  Une telle condition est équivalente à une arête que l'on munit d'une orientation $e_{\mathcal O}$. Avec la polarisation, chaque donnée initiale mène à une marche à gauche en tournant à gauche à chaque sommet rencontré, ou en rebondissant s'il n'y a en ce sommet aucune autre arête. Une marche à gauche est appelée complète si elle couvre toutes les arêtes de $G$ (pas nécessairement dans les deux sens). Nous définissons la valence d'un sommet comme le nombre d'arêtes adjacentes à ce sommet, et la valence d'un graphe comme étant la moyenne des valences de ses sommets. Dans cet article, nous démontrons que si un graphe plongé dans une surface orientée fermée de genre $g$ possède une marche à gauche complète, alors sa valence est d'au plus $1 + \sqrt{6g+1}$. Nous prouvons de plus que ce résultat est optimal pour une infinité de genres $g$ et qu'il est asymptotiquement optimal lorsque $g \to + \infty$. Cela mène à des obstructions pour les plongements de graphes sur une surface. Puisque vérifier si un graphe polarisé possède ou non une marche à gauche complète s'opère en temps au plus $4N$, où $N$ est le nombre d'arêtes (il suffit de le vérifier sur les deux orientations d'une seule  arête donnée), cette obstruction est particulièrement efficace. Ce problème trouve sa motivation dans ses conséquences intéressantes sur ce que nous appellerons ici l'{\it ergodicité topologique} d'un système conservatif, par exemple un système hamiltonien H en dimension deux où l'existence d'une marche complète à gauche correspond à une orbite du système topologiquement ergodique, donc une orbite qui visite toute la topologie de la surface. Nous nous limitons ici à la dimension $2$, mais une généralisation de cette théorie devrait tenir pour des systèmes hamiltoniens autonomes sur une variété symplectique de dimension arbitraire. 
 \end{abstract}
 \end{otherlanguage}

\tableofcontents

\section{Introduction}

Dans cet article, nous établissons des conditions nécessaires optimales pour l'existence d'une «\,longue orbite\,», appelée ici \emph{marche complète}, dans la dynamique des \emph{marches à gauche} canoniquement définie dans tout graphe plongé dans une surface orientée fermée\footnote{Nos résultats s'étendent facilement au cas de graphes plongés dans des surfaces orientées obtenues par épointage de surfaces fermées.}. En employant des termes définis subséquemment dans cette introduction, notre résultat phare s'énonce comme suit\,:\medskip

\noindent \textbf{Théorème.} \emph{Soient $\Sigma$ une surface fermée orientée de genre $g$ et $G \subset \Sigma$ un graphe ordinaire plongé qui est le $1$-squelette d'une décomposition cellulaire de $\Sigma$. Si $G$ (muni de sa polarisation induite) possède une marche à gauche complète, alors la valence moyenne $V = 2A/S$ de $G$ satisfait l'inégalité $V \le 1 + \sqrt{6g+1}$. Cette inégalité est saturée pour une infinité de $g$, notamment pour tous les $g \ge 4$ de la forme $(S-1)(S-3)/6$ où $S \, \equiv \, 7 \mbox{ mod } 12$ est un nombre premier. De plus, cette inégalité est asymptotiquement optimale dans la limite $g \to + \infty$, au sens où pour tout $g$ assez grand, il existe un graphe $G_g \subset \Sigma_g$ tel que $V(G_g) = \sqrt{6g} + o(\sqrt{g})$.}
\medskip

\noindent L'inégalité ci-dessus découle aisément de la formule exprimant la caractéristique de Descartes--Euler de $\Sigma$ comme étant la somme alternée des nombres de Betti de la décomposition cellulaire de $\Sigma$ associée au graphe $G$. La difficulté du Théorème réside ainsi surtout dans ses énoncés d'optimalité. Notre démonstration de la saturation de l'inégalité pour les $g$ listés consiste en la construction de polarisations explicites appropriées pour les graphes complets à $S$ sommets (où $S \equiv 7 \mbox{ mod } 12$ est premier), construction qui repose sur les systèmes de triples de Steiner produits par Skolem et par O'Keefe \cite{S2, O}. L'optimalité asymptotique de l'inégalité résulte quant à elle du fait précédent, d'un résultat de Baker--Harman--Pintz \cite{BHP} sur la répartition des nombres premiers congruents à $7 \mbox{ mod } 12$ et d'une opération de somme connexe sur les paires $(\Sigma_g, G_g)$ qui permettent de définir les graphes appropriés $G_g \subset \Sigma_g$ récursivement sur $g$.

La motivation derrière notre étude des graphes polarisés qui possèdent une marche complète vient de la possibilité de réduire l'étude qualitative des systèmes dynamiques dans une surface à de la combinatoire définie sur un squelette de celle-ci. Par exemple, étant donné un graphe plongé dans une surface orientée qui soit le $1$-squelette d'une décomposition cellulaire de la surface, alors la dynamique des marches à gauche induite dans le graphe peut servir d'approximation pour la dynamique hamiltonienne d'un hamiltonien $H$ défini sur la surface et ayant ledit graphe pour ensemble de niveau. Dans ce cas, la présence d'une marche complète sur le graphe implique l'existence d'une orbite hamiltonienne \emph{topologiquement ergodique}, c'est-à-dire d'une orbite qui intersecte tous les lacets non contractiles de la surface. Le concept d’ergodicité topologique étant intermédiaire entre celui de système intégrable et celui de système ergodique, il s'agit d'un concept intéressant pour l'étude des hamiltoniens rencontrés génériquement. La présence d'une marche à gauche complète dans un graphe de niveau connexe et cellulaire apparaît alors comme une manière utile et efficace d'attester de l'ergodicité topologique de certains hamiltoniens.

\subsection{Notions préalables}

Soit un graphe $G$ (qui contient possiblement des boucles ou de multiples arêtes entre deux sommets). $G$ est \emph{ordinaire} ou \emph{simple} s'il n'a aucun cycle de longueur $1$ ou $2$, et il est \emph{généralisé} sinon. Nous supposerons toujours que $G$ est généralisé, fini et connexe, à moins d'une mention explicite du contraire.

Nous notons $\mathcal{S}(G)$ l'ensemble des sommets de $G$, $S = S(G)$ la cardinalité de $\mathcal{S}(G)$ (c'est-à-dire, le nombre de sommets dans $G$), $\mathcal{A}(G)$ l'ensemble des arêtes (non orientées) de $G$ et $A = A(G)$ la cardinalité de $\mathcal{A}(G)$ (c'est-à-dire, le nombre d'arêtes dans $G$). Nous définissons l'\emph{espace de phase} associé au graphe $G$ comme étant l'ensemble $\mathcal{P}(G)$ des arêtes orientées de $G$, soit encore l'ensemble des demi-arêtes de $G$. De manière équivalente, $\mathcal{P}(G)$ est l'ensemble des paires dont la première composante est un sommet et la seconde est une arête incidente à ce sommet, d'où le nom «\,espace de phase\,», car la position et le moment sont donnés comme conditions initiales.

Une \emph{polarisation} $P$ de $G$ est la donnée pour tout $p \in \mathcal{S}(G)$ d'un ordre cyclique sur les arêtes orientées basées en $p$. Un \emph{graphe polarisé} est une paire $(G,P)$. Etant donné $(G,P)$, nous obtenons une dynamique « des marches à gauche »\footnote{Nous parlerons parfois de « la marche à gauche sur $G$ » pour désigner l'ensemble de la dynamique.} sur l'espace de phase $G$ via l'application $\tau = \tau_{(G,P)} : \mathcal{P}(G) \to  \mathcal{P}(G)$ donnée par $\tau((o,p)) = (p,q)$ où $(p,q)$ est l'arête orientée qui suit l'arête $(p,o)$ dans l'ordre cyclique des arêtes basées en $p$. Géométriquement, la dynamique est donnée par la prescription suivante\,: en arrivant au sommet $p$ via l'arête $(o,p)$, il s'agit de «\,tourner à gauche\,» pour emprunter l'arête $(p,q)$. Si l'on arrive à un sommet qui ne contient que l'arête d'arrivée, la dynamique prescrit de rebondir au sommet et de rebrousser chemin le long de la même arête\footnote{Comme on suit alors la même arête en sens opposé, on retourne au sommet précédent\,; tourner à gauche a l'effet, en ce sommet précédent, de suivre la seconde arête à gauche de celle qui avait mené à ce sommet précédemment. En d'autres termes, un sommet qui n'a qu'une seule arête est un élément neutre et peut moralement être supprimé aussi bien que l'arête qui y mène.}.

Nous notons $\mathcal{F}(G)$ l'ensemble des orbites de la dynamique induite par $P$ et $F = F(G,P)$ la cardinalité de $\mathcal{F}(G)$ (c'est-à-dire, le nombre de marches à gauche). Nous définissons la \emph{caractéristique d'Euler} du graphe polarisé $(G,P)$ par 
\[\chi = \chi(G,P) := S - A + F \, .\]
Il s'avère que ce nombre est pair et vaut au maximum $2$  \cite[Theorem 10.1.2]{HR}. Nous définissons le \emph{genre} du graphe polarisé $(G, P)$ comme étant
\[ \gamma := g(G,P) := 1 - \dfrac{\chi(G,P)}{2} \, . \]

Observons qu'un graphe $G$ plongé dans une surface orientée $\Sigma$ hérite d'une polarisation $P$ induite par l'orientation de $\Sigma$. Nous disons alors que le plongement $\phi : (G, P) \to \Sigma$ est \emph{polarisé}\,; nous dirons parfois simplement que le plongement $\phi : G \to \Sigma$ est polarisé lorsque la polarisation $P$ sur $G$ est sous-entendue. Un plongement $\phi : G \to \Sigma$ est \emph{cellulaire} si $\Sigma \setminus G$ est une union disjointe de $2$-cellules ouvertes.  Le théorème fondamental des plongements polarisés, démontré en toute rigueur par Youngs \cite{Y}, implique que toute polarisation $P$ sur $G$ est induite par un plongement cellulaire de $G$ dans une surface\,:
\begin{thm}[Théorème fondamental]\label{thm-ThmFond}
Soient $(G,P)$ un graphe polarisé et $\Sigma_g$ une surface fermée orientée de genre $g$. Il existe un plongement polarisé $\phi : G \to \Sigma_g$ si et seulement si
\[ \chi(G,P) \ge \chi(\Sigma_g) = 2 - 2g \, .  \]
De plus, l'égalité a lieu si et seulement si le plongement polarisé $\phi$ est cellulaire.
\end{thm}
\noindent L'idée essentielle, que nous emploierons ailleurs dans l'article, consiste à utiliser les marches à gauche déterminées par la polarisation $P$ comme données de recollement des bords de $F$ $2$-cellules ouvertes le long du graphe $G$ afin d'obtenir une surface orientée fermée $\Sigma_{\gamma}$.

Notre attention dans cet article est surtout portée vers les graphes polarisés qui admettent une \emph{marche à gauche complète}, c'est-à-dire une marche à gauche qui emprunte chaque arête (non orientée) de $G$ au moins une fois. (Nous désignerons la marche complète simplement par le sigle \emph{MC}\,; toutes les autres marches à gauche sont dites \emph{non complètes}\footnote{Il existe des graphes polarisés $(G,P)$ qui ont deux marches à gauche  complètes. Pour ceux-ci, nous désignons arbitrairement l'une des deux marches comme MC et l'autre comme MNC.} et désignées simplement par le sigle \emph{MNC}.) L'un des principaux intérêts pour l'étude des marches complètes se trouve dans le fait qu'elles sont étroitement liées à une propriété dynamique que nous nommons \emph{ergodicité topologique}\,:

    \begin{defn}\label{defn-topo_ergo} Soit $W : I \to \Sigma$ un chemin continu dans une surface $\Sigma$ de genre $g \ge 1$. Nous disons que $W$ est {\it topologiquement ergodique} si l'image de $W$ croise l'image de tout lacet non contractile dans $W$. 
    \end{defn}
    
    \begin{prop}  Soient $G \subset \Sigma$ un graphe connexe plongé cellulairement dans une surface fermée orientée de genre $g \ge 1$ et $W$ une marche à gauche complète de $G$ pour la polarisation induite par $\Sigma$. Alors $W$ est topologiquement ergodique.
    \end{prop} 
    
   \begin{proof} Soit $\alpha$ un lacet non contractile dans $\Sigma$. Puisque le graphe est cellulaire et l'image de $\alpha$ est connexe, il en résulte que l'image de $\alpha$ intersecte $G$. Or la marche complète visite toutes les arêtes du graphe, donc $W \cap \alpha$ est non vide.
   \end{proof}

En fait, il y a une certaine réciproque à cette proposition\,: si un graphe plongé $G \subset \Sigma$ possède une marche à gauche $W$ qui est topologiquement ergodique, alors $G$ est plongé cellulairement. De plus, si $G' \subset G$ dénote le sous-graphe parcouru par $W$, alors $W$ est une marche à gauche complète pour $G'$. Nous voyons donc que les graphes cellulaires qui possèdent une marche complète sont les exemples minimaux de graphes plongés admettant une marche topologiquement ergodique.

\subsection{Principaux résultats}

Dans cet article, nous identifions des conditions qu'un graphe plongé cellulairement dans une surface doit satisfaire afin d'admettre une marche à gauche complète. Nos principaux résultats établissent des bornes supérieures -- essentiellement optimales -- sur les valences moyennes que de tels graphes peuvent avoir en fonction du genre $g$ de la surface $\Sigma$ dans laquelle ils sont plongés.

L'existence de telles inégalités est toutefois contrainte par des situations comme celle présentée dans la \cref{fig-multilacets} \,: disposer un nombre arbitraire de lacets à un sommet les uns après les autres n'affecte pas l'existence d'une marche complète, mais permet d'élever la valence \emph{totale} moyenne au-dessus de toute valeur.

Afin de contourner ce problème, nous explorons deux stratégies\,:
\begin{enumerate}
\item La première ne considère que les \emph{graphes homotopiques}, où une condition homotopique vient contraindre les graphes plongés étudiés.

\item La seconde utilise la \emph{valence réduite moyenne} plutôt que la valence totale moyenne, où un compte différent de la valence permet de cerner les contraintes pertinentes imposées par la présence d'une marche complète.
\end{enumerate}
Soulignons que les graphes ordinaires sont homotopiques et que leur valence réduite coïncide avec leur valence totale\,; ces graphes sont donc couverts par les deux stratégies.

  \begin{figure}[h]
  \centering
\includegraphics[width=0.3\textwidth]{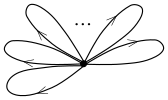}
\caption{}\label{fig-multilacets}
\end{figure}

\subsubsection{Graphes homotopiques}

  \begin{defn}
    Etant donné un entier $S \ge 1$, un \emph{$S$-graphe} $G$ est un graphe formé de $S$ sommets et d'un nombre arbitraire fini d'arêtes et de lacets entre ses sommets. C'est un \emph{monographe} si $S=1$.  Un $S$-graphe plongé dans une surface $\Sigma$ est dit \emph{homotopique} si (i) aucun lacet de $G$ n'est contractile et si (ii) pour chaque paire d'arêtes $e$ et $f$ entre deux mêmes sommets, le lacet $e \overline{f}$ n'est pas contractile.
    \end{defn}

    Un premier intérêt de considérer les monographes cellulaires homotopiques tient au fait que nous avons un théorème optimal les concernant\,:

    \begin{thm}\label{Thm-monographe}  Soit $G$ un monographe homotopique plongé dans une surface orientée fermée de genre $g$. Si $G$ possède une marche à gauche complète, alors le nombre de lacets est au plus $3g$. Ce résultat est optimal (sharp) puisque cette borne est réalisée.
    \end{thm}

\noindent Il convient de comparer ce résultat au nombre maximal de lacets d'un monographe homotopique, sans demander qu'il y ait une marche complète (une preuve est fournie dans \cref{app_invariants} pour le lecteur intéressé):

    \begin{prop}\label{prop-Cg} Soit $C(g)$ le nombre maximal de lacets, tous passant par le même point $p$ d'une surface orientable fermée de genre $g$, tous plongés et disjoints deux à deux sauf en $p$, et représentant des classes non nulles en homologie et toutes différentes. Alors $C(g) = 6g-3$.
    \end{prop}

\noindent Nous constatons donc que l'exigence d'une marche complète divise grosso modo par $2$ le nombre de lacets.

En fait, nous démontrerons le résultat suivant, dont découle le \cref{Thm-monographe} lorsque $S=1$\,:

 \begin{thm}\label{Thm-Sgraphe}  Soient $S \ge 1$ et $G$ un $S$-graphe homotopique plongé dans une surface orientée fermée de genre $g$. Si $G$ possède une marche à gauche complète, alors le nombre d'arêtes est au plus $3g + \left\lfloor \dfrac{3(S-1)}{2}  \right\rfloor$. Ce résultat est optimal puisque cette borne est réalisée.
    \end{thm}

Un second intérêt de considérer les monographes cellulaires homotopiques tient au fait qu'ils peuvent souvent servir de point de départ pour la construction de graphes ordinaires à marche à gauche complète et optimaux grâce à l'application d'opérations appropriées sur les graphes que nous décrivons à la \cref{sec-operations}. Nous exemplifions ce fait dans la \cref{sec-generalise}.

\subsubsection{Valences réduites des graphes généralisés} 

\begin{defn}
Soit un graphe généralisé $G$. La \emph{valence (totale) en un sommet $v \in G$}, notée $V(G)_v$, est la somme des arêtes incidents à $v$. Nous définissons alors la \emph{valence (totale) moyenne de $G$} comme étant la moyenne des valences de ses sommets, $V(G) := (1/S) \sum_{v \in G} V(G)_v = 2A/S$.
\end{defn}

\begin{defn}
Étant donné un graphe généralisé $G$, un \emph{sous-graphe réduit de $G$} est un sous-graphe $G_r$ de $G$ qui a les mêmes sommets que $G$, qui n'a aucun lacet et qui a exactement une arête (issue de $G$) entre n'importe quel deux points distincts $v, w \in G$ qui sont liés par au moins une arête dans $G$. Tous les sous-graphes réduits de $G$ sont des graphes ordinaires et ils ont tous le même nombre d'arêtes\,; le \emph{nombre d'arête réduit de $G$}, noté $A_r(G)$, est le nombre d'arêtes d'un sous-graphe réduit. La \emph{valence réduite en un sommet $v \in G$} est $V_r(G)_v := V(G_r)_v$\,; la \emph{valence réduite moyenne de $G$} est $V_r(G) := (\sum_{v \in G} V_r(G)_v)/S(G) = 2A_r(G)/S(G)$.
\end{defn}

Pour $g \ge 0$, définissons $V(g)$ comme étant le supremum de $V_r(G)$ parmi les $G$ généralisés plongés dans $\Sigma_g$ et ayant une MC. Définissons $V_c(g)$ similairement, mais en imposant que $G \subset \Sigma_g$ soit plongé cellulairement. Il est clair que $V_c(g) \le V(g)$, tandis que le théorème fondamental implique $V(g) = \mathrm{max}_{0 \le h \le g} \, V_c(h)$. Nous montrerons (voir \cref{Prop-Vc-realise}) que $V_c(g)$ est toujours réalisé pour $g \ge 1$, c'est-à-dire qu'il existe un graphe polarisé $(G,P)$ \emph{$g$-optimal}, à savoir un graphe $(G,P)$ qui a une MC, qui satisfaisait $\gamma = g$ et $V(G) = V_c(g)$.

Notre premier théorème présente des bornes sur $V_c(g)$ pour tout $g \ge 0$. Ci-dessous, tandis que $\lceil x \rceil_0$ et $\lceil x \rceil_1$ désignent respectivement les plus petits entiers pair et impair plus grands ou égaux à $x$ et que $\lfloor x \rfloor_0$ et $\lfloor x \rfloor_1$ désignent respectivement les plus petits entiers pair et impair plus petits ou égaux à $x$.

    \begin{thm}\label{Thm_bornes}
Pour $g=0$, $V_c(0) = 3$ et n'est réalisé par aucun graphe. Pour $g \ge 1$, $V_c(g)$ est réalisé et satisfait les inégalités
\[  4 \, \left(1 + \dfrac{1}{3g}\right)^{-1} \le V_c(g) \le b_r(g)      \le b(g) := 1 + \sqrt{1+6g} \, ,\]
où \[ b_r(g) := \mathrm{max} \, \left\{  \, \lfloor{S_0(g)} \rfloor_0 - 1 \, , \, \lfloor{S_1(g)} \rfloor_1 - 1\, , \, 3 + \dfrac{6g-4}{\lceil S_0(g) \rceil_0} \, , \,  3 + \dfrac{6g-3}{\lceil S_1(g) \rceil_1} \right\}  \]
et $S_j(g) := 2 + \sqrt{6g + j}$ pour $j = 0, \, 1$.
      \end{thm}

Nous avons calculé les valeurs de $V_c(g)$ pour $0 \le g \le 5$\,:

\begin{lem}\label{Lem_optimal} Nous avons $V_c(1) = 3 \frac{3}{7}$, $V_c(2) = b_r(2) = 4 \frac{1}{3}$, $V_c(3) = b_r(3) = 5 \frac{1}{7}$, $V_c(4) = b(4) = 6$ et $V_c(5) = b_r(5) = 6 \frac{1}{4}$. Toutes ces valeurs sont réalisées par des graphes ordinaires.
      \end{lem}

Disons qu'un genre $g$ est \emph{maximisé} si $V_c(g) = b_r(g)$ et \emph{super-maximisé} si $V_c(g) = b(g)$. Le lemme précédent suggère qu'il est ardu en général d'établir si un genre $g$ est maximisé, super-maximisé ou pas. Notre prochain théorème montre l'optimalité de nos bornes supérieures\,:

    \begin{thm}\label{Thm_optimal}\label{thm-optimalite}\textcolor{white}{ }
\begin{enumerate}[(a)]
\item $V_c(g) < b_r(g)$ pour une infinité de $g$, notamment tous les $g = 6k^2$ où $k \in \mathbb{Z}_{>0}$.

\item $V_c(g) = b(g)$ pour une infinité de $g$, notamment tous les $g = (S-1)(S-3)/6$ où $S \equiv 7 \, \mathrm{mod} \, 12$ est un nombre premier, ainsi que pour $S = 9$.
\end{enumerate}
      \end{thm}

La première partie de ce résultat repose sur le fait que pour les $g$ cités, la borne $b_r(g)$ est égale à $\lfloor S_0(g) \rfloor_0 - 1$ et à $3 + \dfrac{6g-4}{\lceil S_0(g) \rceil_0}$. Or, les graphes (réduits) qui réalisent la première borne sont des graphes complets sur un nombre pair de sommets, tandis les graphes (réduits) qui réalisent la deuxième borne possèdent plusieurs sommets de valence paire. Puisque ces deux propriétés sont incompatibles, aucun graphe ne réalise $b_r(g)$ si $g = 6k^2$.

La deuxième partie de ce résultat émane de l'observation simple selon laquelle si $g$ est super-maximisé, alors tout graphe $g$-optimal $G$ est isomorphe à un graphe complet $K_S$ sur $S = S(g) = 2 + \sqrt{1+6g }$ sommets, où forcément $S \equiv 1 \mbox{ ou } 3 \, \mathrm{mod} \, 6$. Plus encore, la MC sur $G$ est alors un cycle eulérien, tandis que les MNC déterminent un système de triplets de Steiner sur $\mathcal{S}(G)$, c'est-à-dire qu'elles partitionnent l'ensemble $\mathcal{A}(G)$ en triplets disjoints. La démonstration du théorème consiste donc à construire des polarisations judicieuses pour les graphes complets $K_{S(g)}$ à partir de systèmes de triples de Steiner convenables, tâche que nous accomplissons pour les $g$ mentionnés. Il semble toutefois plausible que la Conjecture suivante tienne\,:

\begin{conj}\label{conj-optimalite}
$V_c(g) = b(g)$ si $g = (S-1)(S-3)/6 \ge 4$ où $S  \equiv 1 \mbox{ ou } 3 \, \mathrm{mod} \, 6$.
\end{conj}

Finalement, en raison du \cref{thm-optimalite}(b), de l'apparition fréquente des nombres premiers congruents à 7 modulo 12 \cite[Theorem 3]{BHP} et de l'opération de somme connexes de surfaces avec graphes plongés, nous montrons que la borne $b(g)$ est asymptotiquement optimale \,:
    \begin{thm}\label{Thm_asymptotique}\label{thm-asymptotique}
\[ V_c(g) = \sqrt{6g} + o(\sqrt{g}) \quad \mbox{ ($g \to + \infty$) }\, . \]
      \end{thm}
\noindent En fait, nous montrerons $V_c(g) = \sqrt{6g} + O(g^{9/20})$. Un raffinement de notre argumentation pourrait fournir un exposant quelque peu meilleur que $9/20$, mais la validité de la \cref{conj-optimalite} impliquerait un terme d'erreur encore meilleur.

\subsection{Quelques problèmes connexes}\label{app}
Nous terminons le survol de cet article en mentionnant quelques problèmes soulevés par notre travail et qui pourront faire l'objet d'investigations futures.

\subsubsection{Premier problème - Conjecture d'optimalité} Il s'agit de mieux comprendre l'ensemble des genres $g \ge 1$ pour lesquels $V_c(g) = b(g)$. La \cref{conj-optimalite} revient à annoncer que la condition nécessaire $g = (S-1)(S-3)/6$ où $S \, \equiv \, 1 \mbox{ ou } 3 \mbox{ mod }6$ avec $S \ge 7$ est en fait une condition suffisante. En raison du \cref{thm-optimalite}(b) et de quelques autres solutions disparates que nous n'avons pas incluses dans cet article, nous croyons que cette conjecture soit correcte.

  \medskip
\subsubsection{Deuxième problème - Relations aux systèmes dynamiques surfaciques}  Bien que nous ayons décidé de présenter notre théorie en termes combinatoires et seulement en dimension deux pour rendre l'article le plus compact possible, les motivations et les conséquences de nos résultats sont liées à deux situations dynamiques que nous présentons succinctement ici :

(1) Jeu de billard. Si $G$ est un graphe plongé dans une surface de genre quelconque, remplaçons chaque sommet par un domaine convexe et chaque arête par un col hyperbolique. Cela donne lieu à un domaine $D$ de la surface. Chaque lancement d'une boule de billard depuis l'un des domaines convexes lui fait parcourir un chemin qui ne peut pas s'engouffrer éternellement dans un col (par hyperbolicité). Un tel chemin donne ainsi lieu à une marche dans le graphe, lorsque projeté sur celui-ci. La théorie que nous présentons ici correspond alors au cas particulier quand la boule traverse chaque col qu'elle croise et ressort de chaque domaine par le col immédiatement à gauche du col par lequel elle est arrivée, ou revient sur elle-même quand il n'y a qu'un col. En ce sens, cet article aborde la partie combinatoire, alors que le jeu de billard classique sur une table convexe aborde la partie analytique. Bien entendu, un domaine régulier quelconque d'une surface n'est pas nécessairement l'épaississement d'un graphe en parties convexes et hyperboliques, mais c'est un premier cas intéressant qui n'est pas hors d'atteinte. 

(2) Dynamique hamiltonienne et théorie du contrôle. Soit $H$ un hamiltonien défini sur une surface orientée de genre quelconque. Il y définit une dynamique dont la nature qualitative a un grand intérêt. Supposons maintenant que $H$ soit générique au sens suivant : il possède un nombre fini de points critiques et chacun d'entre eux est une singularité dont la profondeur est finie. Les fonctions de Morse sont trivialement génériques en ce sens, mais ne sont pas les plus intéressantes car dans le cas général les singularités s'expriment en cartes locales par des polynômes génériques à deux variables de degré arbitraire. Prenant maintenant une valeur critique $c$ de $H$, la préimage de $c$ est, sous des hypothèses raisonnables, un graphe dont la valence en chaque sommet (identifié à un point critique de valeur critique $c$) est donnée par la profondeur du point critique. Dans ce cas, et comme le gradient symplectique de $H$ tourne toujours à gauche, la dynamique de l'hypersurface $H^{-1}(c)$ est approximée $C^0$ par la marche à gauche sur ce graphe $G_H$ ($C^0$ près des points critiques et $C^1$ partout ailleurs). En particulier, toute caractéristique qualitative de cette orbite est contenue dans la marche à gauche sur $G_H$. Comme il est facile de voir si cette marche est complète, et puisque que toute marche complète est topologiquement ergodique lorsque le graphe est plongé cellulairement, l'orbite de $H$ passant par ce point critique est topologiquement ergodique dans ce cas. Moralement, cette orbite visite toute la topologie de la surface. En théorie du contrôle, cela permet de lancer un vaisseau spatial (ici en dimension $2$) en utilisant cette orbite, de n'importe quel point de la surface vers n'importe quel autre point, en y adjoignant un nombre fini de petits hamiltoniens qui correspondent, dans notre exemple, à une contribution des moteurs du vaisseau, avec la plus petite énergie possible. Un autre exemple est la trajectoire d'un électron sur une surface métallique, soumis à un champ magnétique, dont les corrections de trajectoire sont produites par de faibles champs électriques \cite{N, MN}.

   \medskip 
    
\subsubsection{Troisième problème - Généralisation aux dimensions supérieures}   Il y a deux façons, également intéressantes, de généraliser la notion de graphe plongé dans une surface aux variétés symplectiques de dimension arbitraire. La première est de considérer un graphe dans une surface comme un squelette d'hypersurfaces se rencontrant en des sous-variétés coisotropes ou symplectiques. La seconde est de considérer un graphe dans une surface dont le squelette est fait de sous-variétés lagrangiennes sur lequel se rétracte la variété, ce qui arrive dans les variétés de Weinstein.
 \bigskip

\subsection{Relation à la littérature}
Le problème général de plongement de graphes dans des surfaces de genre arbitraire a donné lieu à une littérature abondante et est le sujet du chapitre 10 du livre de Hartsfield et Ringel \cite{HR} et de la monographie de Mohar et Thomassen \cite{MT}.

Le problème que nous abordons dans cet article est celui des marches complètes sur des graphes polarisés, et donc en particulier sur les graphes plongés dans une surface. Ce problème tire son origine de celui, fondamental, de la discrétisation des systèmes dynamiques hamiltoniens. Cet article est, nous l’espérons, une première étape importante vers la résolution de cette discrétisation. 

Or il se trouve que la question des marches (à gauche disons) complètes n’a pas, à ce jour, attiré l’attention des experts en théorie des graphes, peut-être parce que l’accent chez ces experts était davantage porté sur des questions, certes naturelles, qui ne font pas intervenir la dynamique. Il nous fallait donc construire cette nouvelle théorie, plus ou moins à partir de zéro.

C'est ainsi que notre attention s'est naturellement portée vers le problème de la valence moyenne maximale que peut posséder un graphe avec marche complète plongé cellulairement dans une surface donnée. Dans cette approche, la surface a préséance sur le graphe. Il s'agit d'une situation différente de celle que l'on trouve habituellement dans la littérature, où c'est plutôt un graphe qui est donné d'entrée de jeu et pour lequel on cherche une polarisation qui le plonge d'une manière particulière dans une surface (la présence d'une marche complète n'étant cependant jamais une condition imposée).

Par exemple, deux problèmes notables abordés dans la littérature sont ceux du genre minimal et du genre maximal d'une surface dans lequel un graphe peut se plonger cellulairement (nous renvoyons au livre de Mohar et Thomassen \cite{MT} pour des détails sur les résultats obtenus à leur propos). Ces deux problèmes sont généralement différents du nôtre, comme l'illustre le cas du graphe complet à 7 sommets $G = K_7$. En effet, nous avons établi qu'il s'agit du graphe de plus grande valence moyenne à se plonger dans la surface de genre 4 en ayant une marche complète (Propositions 5.10 et 6.1). De plus, $K_7$ n'admet aucun plongement cellulaire avec marche complète dans une surface de genre $g \le 3$, et il n'est pas le graphe de plus grande valence moyenne avec marche complète dans les surfaces de genre $g \ge 5$.   Or les genres minimaux et maximaux des graphes complets sont connus\,: selon Ringel et Young \cite{HR}[Theorem 10.3.6], le genre minimal de $K_7$ est $g = 1$ (un plongement explicite est donné par \cite{HR}[Figure 10.3.7]), tandis que son genre maximal est $g = 7$ en vertu du résultat de Nordhaus, Stewart et White \cite{NSW}.

Observons aussi qu'un graphe polarisé avec une marche complète qui n'a qu'une ou deux marches s'apparente à un graphe supérieurement plongé (traduction libre du terme anglais \emph{upper-embedded graph} introduit par Ringeisen \cite{R}), c'est-à-dire à un graphe dont les plongements cellulaires dans la surface de genre maximal associée ne possèdent qu'une ou deux faces. Évidemment, les graphes polarisés n'ayant qu'une seule marche (forcément complète) coïncident avec les graphes plongés à une seule face. Toutefois, un graphe supérieurement plongé avec deux faces peut n'avoir aucune marche complète : pensons au graphe formé (après une subdivision appropriée des arêtes) par l'union de deux méridiens et d'un cercle de latitude dans le tore. Plus généralement, la question se pose quant à savoir si un graphe polarisé avec marche complète est supérieurement plongeable, c'est-à-dire s'il admet une (autre) polarisation qui le plonge supérieurement dans une (autre) surface. La réponse à cette question est négative\,: les «\,bouquets\,» considérés dans la \cref{fig-multilacets}, ou leurs variantes ordinaires considérées dans \cite[Figure 1]{NSW}, ne peuvent être plongés cellulairement que dans la sphère et ne sont ainsi pas supérieurement plongeables.

Toutefois, la question prend une tournure intéressante lorsqu'elle est posée pour les graphes polarisés avec marche complète qui maximisent la valence moyenne en un genre donné : un tel graphe admet-il une (autre) polarisation qui le plonge supérieurement dans une (autre) surface ? Soulignons que les graphes optimaux explicites des sections 5 et 6 admettent tous, par inspection, un arbre couvrant dont le complément est connexe, de sorte qu'ils sont tous supérieurement plongeables par le théorème principal de cet article de Jungerman \cite{J}. Heuristiquement, le fait qu'un graphe ait une valence moyenne élevée ouvre la voie à ce qu'il ait aussi une arête-connexité élevée ; or, Jungermann (\emph{Loc. cit}) a montré que tout graphe 4-arête-connexe est supérieurement plongeable. Ainsi, il est plausible que cette nouvelle question admette une réponse positive.

\subsection{Structure de l'article}
 Voici la structure de cet article. Dans la \cref{sec-operations}, nous présentons plusieurs opérations sur les graphes avec des marches à gauche (complètes) qui modifient les graphes et les marches de diverses manières contrôlées. Ces opérations sont utiles pour construire divers graphes optimaux qui saturent les bornes indiquées dans nos résultats. Dans la \cref{sec-Sgraphes}, nous prouvons le \cref{Thm-Sgraphe}. La \cref{sec-ineg_struc} établit quelques inégalités fondamentales entre la valence (resp. valence réduite), le nombre de sommets et le genre d'un graphe ayant une marche à gauche complète et aucune marche à gauche ``courte''. La \cref{sec-generalise} contient une preuve du \cref{Thm_bornes} dans la \cref{subsec-thm-bornes} et la \cref{subsec-Lem_optimal} présente les exemples optimaux attestés par le \cref{Lem_optimal}. Dans la \cref{sec-thm-optimalite}, nous prouvons le \cref{thm-optimalite}, tandis que la \cref{sec-thm-asymptotique} contient une preuve du \cref{thm-asymptotique}. L'\cref{app_invariants} présente un calcul de l'invariant $C(g)$ pour le lecteur intéressé.

   \bigskip \medskip

   \noindent
   \textit{Remerciements.} Nous sommes reconnaissants à Steven Boyer de nous avoir communiqué, dès le début de ce travail, une preuve simple et lumineuse d'un résultat classique sur le nombre maximal de lacets plongées disjoints et homotopiquement différents dans surface de genre arbitraire, et à Fran\c{c}ois Bergeron pour une conversation sur la combinatoire algébrique. Nos remerciements vont à Jacob Fox et à Yakov Eliashberg pour des discussions fructueuses. Nous sommes reconnaissants à Thomas Parker de nous avoir suggéré d'étudier la topologie de 1-courants sur l'espace des graphes. Nous sommes reconnaissants à Eliane Cody d'avoir suggéré et étudié cette théorie pour des graphes infinis, en particulier pour le H-tree. Bien que cela n'entre pas dans le contexte de cet article, nous la remercions pour cette étude et sa contribution indirecte. Le second auteur remercie le Département de mathématique de l'Université Stanford pour son soutien lors de la présentation à l'automne 2022 d'une suite de conférences ``Distinguished Lecture Series" portant sur cet article. Nous remercions également le référé pour de nombreuses suggestions sur les relations éventuelles de ce travail avec certains modèles de surfaces aléatoires et sur un traitement plus combinatoire des marches à gauche. Bien que ces idées nous semblent intéressantes, il ne nous a pas semblé possible de les ajouter dans cet article dans un temps raisonnable.

\section{Opérations sur les graphes polarisés}\label{sec-operations}

Nous décrivons plus en détails les opérations mentionnées dans l'introduction. Bien que ces opérations puissent être définies au niveau des graphes polarisés abstraits, il est utile d'imaginer qu'elles opèrent sur un graphe plongé dans une surface. Pour chacune d'elles, nous mentionnons les conditions pour que l'opération maintienne la présence d'une marche à gauche complète, pour qu'elle préserve la cellularité du plongement et pour qu'elle conserve l'ordinarité du graphe. Nous évoquons aussi l'impact de l'opération sur la valence moyenne du graphe.\\

   \noindent
   \textbf{Contraction (Blow-down)} : Si $G$ est un graphe et $G' \subset G$ un sous-graphe, cette opération consiste à contracter toutes arêtes de $G'$ et à contracter les sommets de $G'$ en un seul sommet. \smallskip

Si $G$ est plongé dans une surface, et si $G'$ est un arbre, la contraction de $G'$ ne change pas la topologie de la surface. Sinon, dans le cas général, la contraction $G/G'$ est un graphe sur une nouvelle surface singulière (à cusps) qui correspond exactement à la partie topologique du théorème de compacité de Gromov. Evidemment il suffit de répéter la contraction d'une seule arête, autant de fois qu'il le faut pour épuiser $G'$.

Si $G'$ est un arbre, le nouveau graphe possède encore une marche à gauche complète et le caractère cellulaire du plongement est préservé. Si la valence moyenne de départ est supérieure à $2$, alors la contraction de cet arbre augmente la valence moyenne. L'ordinarité du graphe n'est généralement pas préservée.

     \smallskip
   \noindent
   \textbf{Eclatement (Blow-up)} : Si $v$ est un sommet d'un graphe polarisé, et si $K_v$ est une coupure dans l'ordre cyclique des arêtes incidentes à $v$, l'éclatement de $K_v$ introduit une arête $e_{K_v}$ et un nouveau sommet $v_{K_v}$. L'arête relie $v$ à $v_{K_v}$ et les arêtes incidentes à $v$ sont partitionnées entre $v$ et $v_{K_v}$ selon la coupure. \smallskip

Cette opération préserve les marches (et en particulier les marches complètes) si dans la partition $K_v$, la première arête dans l'ordre cyclique est sortante et la dernière entrante. Cette opération doit respecter cette condition. Itérer cette opération en $v$ revient à partitionner les arêtes en plusieurs ensembles respectant l'ordre cyclique et la condition sortant-entrant.

En effectuant une suite d'éclatements en des sommets créés lors d'éclatements précédents, l'effet net est l'éclatement d'un sommet en un arbre dont toutes les nouvelles branches sont parcourues dans les deux sens par la marche complète. Nous parlerons donc d'\emph{éclatement élémentaire} lorsqu'une seule arête est créée. Voir la \cref{fig-eclat} pour une représentation graphique.
   \begin{figure}[h]
\centering
\begin{minipage}{.5\textwidth}
  \centering
  \includegraphics[width=.8\linewidth]{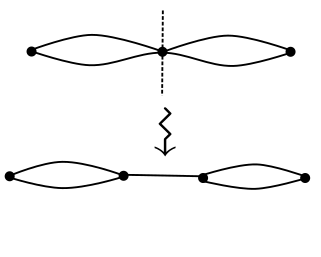}
\end{minipage}%
\begin{minipage}{.5\textwidth}
  \centering
  \includegraphics[width=.9\linewidth]{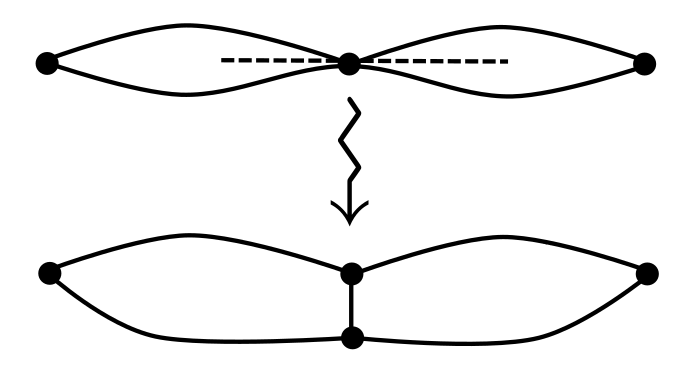}
\end{minipage}
\caption{Illustration de deux façons de faire l'éclatement} \label{fig-eclat}
\end{figure}  
\\
   \noindent
  \textbf{Chirurgie (Surgery)} : Si $e, f$ sont deux arêtes incidentes à un sommet $v$, de sorte qu'elles soient successives dans l'ordre cyclique, et que la marche entre par $e$ et ressorte à gauche par $f$, la chirurgie consiste à unir les arêtes $e$ et $f$, et donc à ne plus passer par $v$. Voir la \cref{fig-chirurgie} pour une représentation graphique. \smallskip

Evidemment, cette opération n'a aucun sens si l'une ou les deux arêtes sont parcourues dans les deux sens. Dans de tels cas il faut alors dédoubler les arêtes parcourues dans les deux sens (opération décrite plus bas) avant de faire cette la chirurgie.

La préservation d'une marche à gauche complète est claire. Cette opération préserve l'ordinarité d'un graphe et décroît la valence moyenne (si la valence moyenne de départ est supérieure à $2$). En général, la cellularité du graphe n'est pas préservée.  \\

   \noindent
   \textbf{Subdivision} : Pour une arête $e$ de $G$, nous ajoutons un sommets $v'$ au milieu de $e$ et produisons ainsi deux nouvelles arêtes incidentes à $v'$. \smallskip

Cette opération préserve assurément la présence d'une marche à gauche, ainsi que la cellularité et l'ordinarité d'un graphe. Si la valence moyenne de départ est supérieure à $2$, alors cette opération décroit la valence moyenne. \\

   \noindent
   \textbf{Arête parallèle (Parallel edge)} : Soit $G$ un graphe polarisé, disons plongé dans une surface. Supposons que la marche à gauche complète suive une chaîne $C$ d'arêtes orientées distinctes $e_1,  \ldots, e_{k-1}$ qui relie les points $p_1$ et $p_k$ (possiblement égaux) et telle que toutes les arêtes $\overline{e_1}, \ldots, \overline{e_{k-1}}$ soient aussi parcourues par la MC (pas forcément consécutivement). Nous pouvons alors introduire une nouvelle arête $e$ entre $p_1$ et $p_k$ tout juste à la gauche de $C$. \smallskip

Le graphe polarisé ainsi obtenu possède encore une marche complète, qui diffère de la MC originelle précisément du fait qu'elle suit l'arête orientée $e$ plutôt que la chaîne $C$. Cette opération préserve le caractère cellulaire d'un plongement. Si aucune arête ne lie directement $p_1$ et $p_k$, alors cette opération préserve l'ordinarité du graphe. La valence moyenne est aussi augmentée.

Un cas particulier de cette opération est le "doubling trick" que voici:
   
    \smallskip
   \noindent
   \textbf{Dédoublement (Doubling trick)} : Si $G$ est un graphe généralisé polarisé, disons plongé dans une surface, et si une arête $b$ du graphe, reliant $u$ à $v$, est parcourue deux fois dans la marche, donc dans les deux sens, le doubling trick consiste à introduire une nouvelle arête plongée $b'$ $C^1$-près de $b$. Voir la \cref{fig-doub} pour une représentation graphique. \smallskip

Il s'agit d'un cas particulier de l'addition d'une arête parallèle, la chaîne $C$ étant ici prise égale à l'arête orientée $\bar{b}$, qui se voit alors remplacée par une arête orientée $b'$ légèrement déplacée vers la gauche de $\bar{b}$. La marche à gauche visitera alors $b$ dans un sens, et $b'$ dans l'autre sens. Cette opération préserve les marches, elle est nécessaire pour l'opération de chirurgie. \\
 \begin{figure}[h]
\centering
 \begin{subfigure}{.5\textwidth}
  \centering
  \includegraphics[width=.7\linewidth]{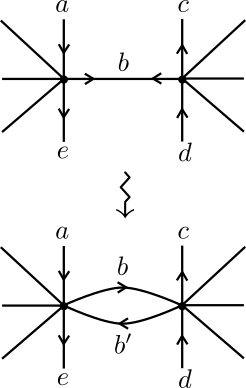}
  \caption{Le ``doubling trick''} \label{fig-doub}
\end{subfigure}%
\begin{subfigure}{.5\textwidth}
  \centering
  \includegraphics[width=0.3\linewidth]{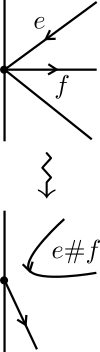}
  \caption{La chirurgie} \label{fig-chirurgie}
\end{subfigure}
\caption{}
\end{figure}  
   \noindent
   \textbf{Somme connexe (Connected sum)} : Soient $\Sigma_1$ et $\Sigma_2$ deux surfaces compactes orientées et $G_j \subset \Sigma_j$ ($j=1,2$) des graphes plongés cellulairement avec MC. Soient $F_j \subset \Sigma_j$ la $2$-cellule ouverte bordée par la MC de $G_j$. Effectuons la somme connexe $F_1 \# F_2$ de façon à obtenir une surface compacte orientée $\Sigma_1 \# \Sigma_2$. Traçons ensuite une arête $e'$ dans $F_1 \# F_2$ qui relie  $v_1 \in G_1$ à $v_2 \in G_2$\,; nous obtenons ainsi un graphe $G' \subset \Sigma_1 \# \Sigma_2$. Finalement, contractons $e'$ pour obtenir le graphe $G_1 \# G_2 \subset \Sigma_1 \# \Sigma_2$. \smallskip

Puisque $(F_1 \# F_2)\setminus e$ est un disque ouvert, les graphes $G'$ et $G_1 \# G_2$ sont cellulaires dans $\Sigma_1 \# \Sigma_2$. Les MNC de $G_1$ et de $G_2$ sont des MNC de $G'$. $G'$ possède une MC qui consiste à parcourir entièrement la MC de $G_1$ à partir de $v_1$ (en empruntant d'abord l'arête de $G_1$ qui suit l'arête $e$ dans l'ordre cyclique déterminée par la polarisation induite par l'orientation de $\Sigma_1 \# \Sigma_2$), puis à suivre $e$ vers $v_2$, à parcourir entièrement la MC de $G_2$ à partir de $v_2$ (en empruntant d'abord l'arête qui suit $e$ dans l'ordre cyclique déterminée par la polarisation), puis à revenir à $v_1$ via $e$. Il en résulte que $G_1 \# G_2$ possède aussi une MC. Voir la \cref{fig-SommeConnexe} pour une représentation graphique.

En tant que graphe abstrait, $G_1 \# G_2$ est obtenu de l'union disjointe $G_1 \sqcup G_2$ en identifiant les sommets $v_1$ et $v_2$. Soit $v_{\#} \in G_1 \# G_2$ le sommet correspondant. Alors $S(G_1 \# G_2) = S(G_1) + S(G_2) - 1$, $V_r(G_1 \# G_2)_{v_{\#}} = V_r(G_1)_{v_1} + V_r(G_2)_{v_2}$ et $V_r(G_1 \# G_2)_{v} = V_r(G_j)_{v}$ pour tout $v \neq v_{\#}$ issu de $G_j$ (les mêmes relations tiennent pour la valence totale). Il s'ensuit l'identité
\begin{align}
\tag{$\#$}  \notag  V_r(G_1 \# G_2) &= V_r(G_1) \, \dfrac{S(G_1)}{S(G_1) + S(G_2) - 1} + V_r(G_2) \, \dfrac{S(G_2)}{S(G_1) + S(G_2) - 1}   \, .
\end{align}
\noindent Il est clair que $G_1 \# G_2$ est ordinaire si $G_1$ et $G_2$ le sont.

\begin{figure}[h]
\centering
\includegraphics[angle=90, width=1.2\textwidth]{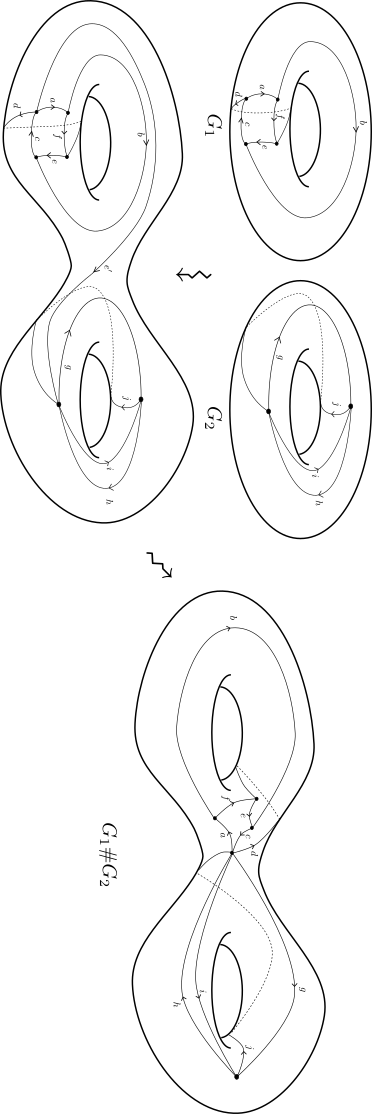}
\caption{Somme connexe}\label{fig-SommeConnexe}
\end{figure}


\section{Résultats sur les graphes homotopiques} \label{sec-Sgraphes}

Dans cette section, nous prouvons le \cref{Thm-Sgraphe} à travers les propositions \ref{Prop-HomGenusConstraint} et \ref{Prop-HomGenusOptimal} ci-dessous. Pour un graphe polarisé $(G,P)$, nous définissons $g_{hom}(G,P)$ comme étant le minimum $g \in \N$ tel qu'il existe un plongement polarisé de $(G,P)$ dans $\Sigma_g$ une surface de genre $g$ de telle sorte que le graphe plongé soit homotopique. Pour $i \in \Z_{>0}$, nous désignons par $\ell_i:= \ell_i(G,P)$ le nombre de marches à gauche dans $(G,P)$ de longueur $i$.
\begin{prop}\label{Prop-HomGenusComp}
Pour tout graphe polarisé $(G,P)$,
\begin{align*}
\gamma(G,P) +  \frac{\ell_1 + \ell_2}{2} &\leq g_{hom}(G,P).
\end{align*}
\end{prop}
\begin{proof}
Soit $\phi : (G,P) \rightarrow \Sigma_g$ un plongement polarisé de $(G,P)$ dans une surface de genre $g= g_{hom}(G,P)$ tel que $\phi(G)$ soit un graphe homotopique (par la suite, nous écrivons simplement $G$ pour désigner $\phi(G)$). Soit $W$ une marche à gauche de longueur $1$ ou $2$ et soit $N$ un $\epsilon$-voisinage de $W$ dans $\Sigma_g$. Pour $\epsilon >0$ suffisamment petit, la composante connexe de la frontière de $N$ qui se trouve à gauche de $W$ (par rapport à l'orientation donnée par la traversée des arêtes dans la direction décrite par $W$) est l'image d'une courbe plongée $\alpha$ qui est homotope à $W$. Notons que $\alpha$ n'est pas contractile, puisque $G$ est homotopique par hypothèse.

En coupant le long de $\alpha$ et en collant deux disques le long des deux composantes de bord, alors selon que $\alpha \subset \Sigma$ soit non séparante ou séparante, nous obtenons soit une surface fermée $\Sigma'$, soit une union disjointe de deux surfaces fermées $\Sigma'$ et $\Sigma''$ avec disons $G \subset \Sigma'$. Dans les deux cas de figure, la surface $\Sigma'$ a genre $g-1$ ou moins\,: dans le premier cas, cela résulte du fait que $\Sigma$ s'obtient de $\Sigma'$ par l'ajout d'une anse\,; dans le second cas, c'est parce que $\Sigma$ est la somme connexe de $\Sigma'$ et $\Sigma''$ et que $\Sigma''$ n'est pas une sphère puisque $\alpha$ n'est pas contractile.

Bref, nous obtenons un plongement de $(G,P)$ dans une surface fermée de genre au plus $g-1$ telle que $W$ et possiblement une seule autre marche $W'$ de longueur $1$ ou $2$ soient contractiles dans $\Sigma'$ (la marche $W'$ serait homotope à la deuxième composante de bord de $\Sigma \setminus \alpha$ dans le cas où $\alpha$ est non séparante). Il est clair que nous pouvons répéter ce processus tant qu'il reste des marches de longueur $1$ ou $2$ qui ne sont pas contractiles, et ainsi le genre peut être réduit d'au moins $\frac{\ell_1 + \ell_2}{2}$. Puisque $\gamma(G,P)$ est une borne inférieure pour le genre de toute surface dans laquelle $(G,P)$ se plonge, l'affirmation s'ensuit.
\end{proof}
Pour $S \in \N$, posons $\pi(S)=0$ si $S$ est pair et posons $\pi(S)=1$ si $S$ est impair.
\begin{prop}\label{Prop-HomGenusConstraint}
Soit $G$ un $S$-graphe homotopique plongé dans une surface $\Sigma$ de genre $g$. Si $G$ admet une marche à gauche complète, alors $G$ possède au plus $3g + \left\lfloor \dfrac{3(S-1)}{2} \right\rfloor$ arêtes, c'est-à-dire que
\[  V(G) \le 3 + \dfrac{6g - 4 + \pi(S)}{S} \, . \]
\end{prop}
\begin{proof}
Puisque $G$ a un nombre entier d'arêtes, il suffit de montrer que ce nombre vaut tout au plus $3g + 3(S-1)/2$. Supposons en vue d'une contradiction qu'il existe un $S$-graphe polarisé $(G,P)$ avec au moins $3g+3(S-1)/2 + 1$ arêtes, qui admet une marche à gauche complète et tel qu'il existe un plongement polarisé $\phi : (G,P) \rightarrow \Sigma_g$ tel que $\phi(G)$ est homotopique. D'après la définition de $\gamma=\gamma(G,P)$, nous avons $2-2\gamma = S - A +F$ et $F \leq 1 + \ell_1 + \ell_2 - \frac{A-\ell_1-2\ell_2}{3}$, puisque $(G,P)$ est un graphe admettant une marche à gauche complète, chaque marche non complète de longueur $i$ nécessite exactement $i$ arêtes et chaque arête appartient à au plus une marche non complète. Ceci implique 
\begin{align*}
3g + \dfrac{3(S-1)}{2} +1 \leq & \, A \leq 3 \gamma + \dfrac{3(S-1)}{2} + \ell_1 +\frac{\ell_2}{2},
\end{align*}
d'où $3(g- \gamma)+1 \leq \ell_1 + \frac{\ell_2}{2}$. D'après la définition de $g_{hom}=g_{hom}(G,P)$, nous avons $g_{hom} \leq g$ de sorte que $3(g_{hom}- \gamma)+1 \leq \ell_1 + \frac{\ell_2}{2}$. La \cref{Prop-HomGenusComp} implique
\begin{align*}
3\frac{\ell_1 + \ell_2}{2} +1 &\leq \ell_1 + \frac{\ell_2}{2} \\
 \Leftrightarrow \; \; \frac{\ell_1}{2} + \ell_2 &\leq -1,
\end{align*}
ce qui est en contradiction évidente avec le fait que $\ell_1$ et $\ell_2$ sont des entiers naturels. 
\end{proof}

\begin{prop}\label{Prop-HomGenusOptimal}
Pour tout $S, g \geq 1$, il existe un $S$-graphe homotopique plongé cellulairement dans $\Sigma_g$,  qui admet une marche à gauche complète et qui a \[ 3g + \left\lfloor \dfrac{3(S-1)}{2}  \right\rfloor \; \mbox{ arêtes,}\]
de sorte que
\[  V(G) = 3 + \dfrac{6g - 4 + \pi(S)}{S} \, . \]
\end{prop}
\begin{proof}
Pour tout $g \geq 1$, nous commençons par construire la surface $\Sigma_g$ de genre $g$ comme surface de translation, c'est-à-dire au moyen d'un polygone $P=P_{4g}$ à $4g$ côtés cycliquement identifiés $a_1,b_1, \ldots, a_g, b_g, \bar{a}_1,\bar{b}_1, \ldots, \bar{a}_g,\bar{b}_g$ et d'une application de quotient $q: P \rightarrow \Sigma_g$ qui identifie $a_i$ et $b_i$ avec $\bar{a}_i$ et $\bar{b}_i$ respectivement via des homéomorphismes inversant l'orientation, pour $i=1,\ldots,g$. Il est utile de dénoter les sommets du polygone par $p_1, p_2, \dots, p_{4g}$, où $a_1 = p_1p_2$, $b_1 = p_2p_3$, etc. Bien sûr, tous ces $p_i$ sont identifiés par $q$ au même sommet $v$. Observons que le monographe $G_0 \subset \Sigma_g$ résultant a $2g$ lacets qui appartiennent à des classes d'homotopie (et même d'homologie) distinctes et qu'il admet une seule marche à gauche (forcément complète).

Ensuite, soit $S'$ le plus grand entier impair $\le S$. Nous subdivisons l'arête $a_1$ en $S'$ arêtes, introduisant de ce fait $S'-1$ sommets $p'_1, \ldots, p'_{S'-1}$ entre $p_1$ et $p_2$, dans cet ordre. Dans $\Sigma_g$, tous ces points sont distincts. (Si $S = 1$, aucun sommet n'est alors ajouté.) Posons $p'_0 = p_1$, $p'_{S'} = p_2$ et $p'_{S'+1} = p_3$.   Pour $j = 1, \ldots, (S'+1)/2$, considérons aussi les chaînes orientées $C_j = p'_{2j-2}p'_{2j-1}p'_{2j}$\,; ce sont des segments de la marche complète sur $G_0$ et nous pouvons donc ajouter une arête parallèle $e_j$ à chaque $C_j$. L'effet net est d'obtenir une arête $e_j$ liant les points $p'_{2j-2}$ et $p'_{2j}$ pour tout $j = 1, \ldots, (S'+1)/2$. Similairement, pour chaque $k = 1, \ldots, g-1$, nous relions les sommets $p_{2k+1}$ et $p_{2k+3}$ par une arête $d_k$ parallèle à la chaîne $a_{k+1}b_{k+1}$. (Voir la \cref{fig-Prop-5.3-Fig} pour une représentation de cette construction).

Par projection par $q$, nous obtenons ainsi un graphe plongé $G_1 \subset \Sigma_g$ avec  $S'$ sommets et $3g + 3(S'-1)/2$ arêtes et qui admet une marche complète. Nous affirmons que $G_1$ est homotopique. D'une part, si $S' > 1$, alors chaque arête $q(e_j)$ est l'unique arête dans $G_1$ à relier ses extrémités, qui sont des sommets distincts. Si $S'=1$, alors la seule arête $q(e_1)$ relie les points $q(p_1) = q(p_3) = v$ et représente la classe d'homologie $[q(a_1)] + [q(b_1)]$. D'autre part, les arêtes $q(f_k)$ relient toutes le même sommet $v$, mais elles représentent les classes d'homologie distinctes $[q(a_{k+1})] + [q(b_{k+1})]$. Dans tous les cas, les arêtes $q(e_j)$ et $q(f_k)$ ajoutées au graphe $G_0$ pour obtenir $G_1$ appartiennent à des classes d'homologie distinctes entre elles et distinctes des classes d'homologie des arêtes présentes dans $G_0$.

Si $S' = S$, alors $G := G_1$ satisfait l'affirmation. Si $S = S'+1$, alors nous pouvons subdiviser n'importe quelle arête pour obtenir un graphe $G := G_2$ ayant $S$ sommets et $3g + 3(S'-1)/2 + 1 = 3g + \left\lfloor 3(S-1)/2  \right\rfloor$ arêtes. Il est clair que $G$ admet une marche à gauche complète, qu'il est homotopique et qu'il est plongé cellulairement.

\end{proof}

\begin{figure}[h]
\centering
 \begin{subfigure}{.49\textwidth}
  \centering
  \includegraphics[width=\linewidth]{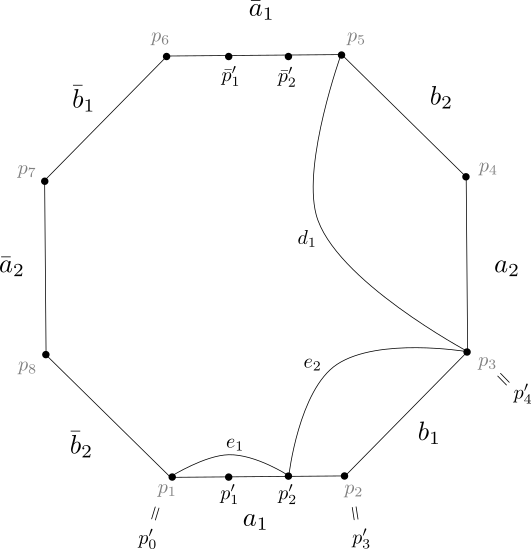}
\caption{Cas $g=2$, $S=S'=3$ de la construction dans la preuve de la \cref{Prop-HomGenusOptimal}}\label{fig-Prop-5.3-Fig}
\end{subfigure}%
\hspace{4pt}
\begin{subfigure}{.49\textwidth}
  \centering
  \includegraphics[width=0.9\linewidth]{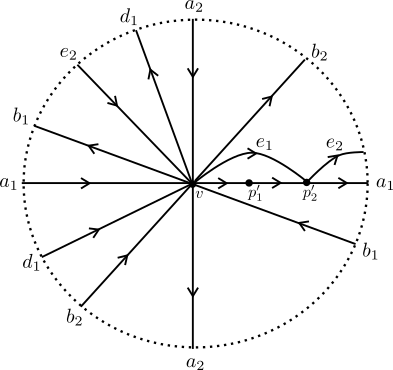}
\vspace{12pt}
  \caption{Un voisinage du sommet $v=q(p_1)$ dans $G_1$ dans la preuve de la \cref{Prop-HomGenusOptimal}. Les flèches indiquent la direction de la marche complète vers la gauche.}
\end{subfigure}
\caption{}
\end{figure}  

Le cas $S=1$ nous sera particulièrement utile\,: la construction décrite ci-dessus fournit alors un monographe à $3g$ lacets, obtenu à partir du polygone de côtés  $a_1,b_1, \ldots, a_g, b_g, \bar{a}_1,\bar{b}_1, \ldots, \bar{a}_g,\bar{b}_g$ en tirant $g$ diagonales $d_1, \ldots, d_g$ où $d_j$ relie le point source de $a_j$ au point cible de $b_j$. Nous appellerons ces monographes précis \emph{monographes optimaux standard}.


\section{Inégalités structurelles}\label{sec-ineg_struc}

\subsection{Inégalités pour la valence totale}

Soit un graphe généralisé polarisé $(G, P)$ ayant une marche à gauche complète et satisfaisant la condition suivante\,:\vspace{6pt}

\begin{itemize}
\item[ (C)] Le graphe polarisé $(G,P)$ n'a pas de marche à gauche de longueur 1 ou 2.\\
\end{itemize}

\noindent Observons que les graphes ordinaires satisfont la condition (C). Afin d'alléger les notations, posons $S = S(G)$, $A = A(G)$ et $V := V(G) =  (\sum_{v \in G} V(G)_v)/S$ (où $V$ dénote la valence totale).

Puisque le nombre de demi-arêtes dans $G$ est $2A$ et est aussi $\sum_{v \in G} V(G)_v$, ceci implique
\[ \tag{$\spadesuit$}   A = \dfrac{V \, S }{2} \, . \]

\noindent Chacune des $2A$ arêtes \emph{orientées} de $G$ appartient à une unique orbite de la marche à gauche sur $(G,P)$. La MC a longueur au moins $A$, de sorte que toutes les MNC couvrent ensemble au maximum $A$ arêtes. La condition (C) implique que chacune de ces MNC a longueur $3$ ou plus. Tout ceci implique donc
\[ \tag{$\heartsuit$}   F \le  1 + \dfrac{A }{3} \, . \]

\noindent Par définition du genre $\gamma := \gamma(G,P)$, nous avons $2 - 2\gamma = S - A + F$. Cette égalité, ($\spadesuit$) et ($\heartsuit$) donnent donc, après redistribution des termes et utilisation à nouveau de ($\spadesuit$)\,:
\[   V \le 3 + \dfrac{6\gamma - 3}{S}  \, . \]
\noindent Une inégalité un peu plus forte tient\,: par l'inégalité précédente et ($\spadesuit$), nous avons
\[ V = \dfrac{2A}{S} \le  3 + \dfrac{6\gamma-3}{S} = \dfrac{3(S-1) + 6\gamma}{S} \; \Leftrightarrow \; 2A \le 3(S-1) + 6\gamma .  \]
Puisque cette dernière inégalité compare des entiers et que $2A$ est pair, nous pouvons renforcer cette inégalité comme suit, où $\pi(n) \in \{0,1\}$ dénote la parité du nombre entier $n$\,:
\[  2A \le 3(S-1) + 6\gamma - 1 + \pi(S)  \,  ,  \]
ce qui donne
\[ \tag{$\diamondsuit$} V \le 3 + \dfrac{6\gamma - 4 + \pi(S)}{S}  \, . \]\smallskip

\subsection{Inégalités pour la valence réduite}

Observons que par définition de la valence réduite, nous avons $V_r(G)_v \le S(G) - 1$ pour tout sommet $v \in G$. Ainsi,  nous avons
\[ \tag{$\clubsuit_r$} V_r \le S-1 \, . \]
Similairement, puisque $V_r(G)_v \le V(G)_v$ pour tout $v \in G$, nous avons $V_r \le V$ et donc ($\diamondsuit$) implique
\[\tag{$\diamondsuit_r$} V_r \le 3 + \dfrac{6\gamma - 4 + \pi(S)}{S} \, .\]
Finalement, nous avons
\[ \tag{$\spadesuit_r$}   A_r = \dfrac{V_r \, S_r }{2} \, . \]


\section{Résultats pour les graphes généralisés} \label{sec-generalise}

\subsection{Réduction aux graphes sans courte marche}

Nous débutons d'abord par une procédure de réduction du graphe polarisé.

\begin{lem}\label{lem-reduction}
Soit $(G,P)$ un graphe généralisé polarisé ayant une marche complète et $S \ge 3$. Alors il existe un graphe généralisé polarisé $(G', P')$ qui a une marche complète et qui satisfait la condition (C), qui ne diffère de $G$ que par le retrait de certaines arêtes, tel que $\gamma(G', P') = \gamma(G,P)$ et tel que $V_r(G')_v = V_r(G)_v$ pour tout $v \in \mathcal{S}(G) = \mathcal{S}(G')$.
\end{lem}

\begin{proof}
Plongeons $(G,P)$ cellulairement dans la surface $\Sigma$ de genre $g = \gamma(G,P)$.

Toute marche à gauche de longueur $1$ de $(G,P)$ est un lacet basé en $v \in G$ qui délimite un disque topologique dans $\Sigma$\,: en contractant la lacet à travers le disque vers le sommet $v$, nous obtenons un nouveau graphe $G'' \subset \Sigma$ qui a exactement les mêmes sommets que $G$. Puisque le lacet contracté ne participait pas à la valence réduite $V_r(G)_v$, nous avons $V_r(G'')_p = V_r(G)_p$ pour tout $p \in \mathcal{S}(G) = \mathcal{S}(G'')$. Les marches à gauche de la polarisation $P''$ induite par $\Sigma$ sur $G''$ sont les mêmes que celles de $(G,P)$, sauf (i) pour la marche de longueur $1$ qui a été contractée et qui n'existe plus dans $(G'',P'')$ et (ii) pour la marche complète de $(G,P)$ qui se voit raccourcie d'une arête dans $(G'', P'')$ et qui est complète pour $(G'', P'')$. Il est clair que $G''$ est plongé cellulairement, donc $\gamma(G'', P'') = g$ \,; alternativement, cela résulte des relations $S(G'') = S(G)$, $A(G'') = A(G)-1$ et $F(G'', P'') = F(G,P) - 1$. Par récurrence, nous nous réduisons à un graphe $(G'',P'')$ qui n'a pas de marche à gauche de longueur $1$. 

Toute marche à gauche de longueur $2$ est une paire d'arêtes $(e,f)$ reliant des sommets $v, w \in G''$ et qui délimitent un disque topologique $D$ (un bigone) dans $\Sigma$. Afin d'être précis, supposons que le bord orienté de $D$ est $\partial D = f + \bar{e}$. Nous effectuons alors l'opération inverse à l'opération de dédoublement\,: via une homotopie dans ce disque relative aux sommets $v$ et $w$, nous «\,contractons\,» le disque et homotopons ainsi $f$ sur $e$. Le résultat est un graphe $G' \subset \Sigma$ qui ne diffère de $G''$ que par le retrait de $f$. Compte tenu de $e$, l'arête $f$ ne participait pas aux valences réduites $V_r(G'')_v$ et $V_r(G'')_w$, donc nous avons $V_r(G')_p = V_r(G'')_p$ pour tout $p \in \mathcal{S}(G) = \mathcal{S}(G')$. Les marches à gauche de la polarisation $P''$ induite par $\Sigma$ sur $G''$ sont les mêmes que celles de $(G'',P'')$, sauf (i) pour la marche $[f, \bar{e}]$ qui a été contractée et qui n'existe plus dans $(G',P')$ et (ii) pour la marche complète de $(G'',P'')$ qui, au lieu de passer par $f$, passe par $e$ dans $(G', P')$ et est donc encore complète. Il est clair que $G'$ est plongé cellulairement, donc $\gamma(G', P') = g$ \,; alternativement, cela résulte des relations $S(G') = S(G'')$, $A(G') = A(G'')-1$ et $F(G', P') = F(G'',P'') - 1$. Par récurrence, nous nous réduisons à un graphe $(G',P')$ qui n'a pas de marche à gauche de longueur $1$ ou $2$, bref qui satisfait la condition (C).
\end{proof}

Ainsi, il nous suffit de démontrer le \cref{Thm_bornes} sous l'hypothèse additionnelle que $G$ satisfait la condition (C), ce qui nous permet d'avoir recours aux inégalités structurelles de la \cref{sec-ineg_struc}.

\bigskip


\subsection{\Cref{Thm_bornes}}\label{subsec-thm-bornes} 
La démonstration est scindée en plusieurs propositions.
    \begin{prop}  $V_c(0) = 3$ et ce supremum n'est réalisé par aucun graphe.
    \end{prop} 
\begin{proof}
D'abord, ($\diamondsuit$) implique $V_r \le V < 3$ pour tout graphe plongé cellulairement dans la sphère et qui a une MC, de sorte que $V_c(0) \le 3$ et $V(G)<3$ pour tout graphe $G$ plongé cellulairement dans la sphère. Ensuite, pour tout $n \in \mathbb{Z}_{> 0}$, considérons d'abord le monographe à $n$ lacets, puis subdivisons chaque lacet en trois arêtes afin d'obtenir un graphe ordinaire $G_n$. Il est clair que $G_n$ se plonge dans la sphère et qu'il possède une MC. Puisque $G_n$ a un sommet de valence $2n$ et $2n$ sommets de valence $2$, nous calculons $V(G_n) = 6n/(1+2n)$. En prenant la limite quand $n \to + \infty$, nous déduisons que $V_c(0) \ge 3$. Bref, $V_c(0) = 3$ et ce supremum n'est pas réalisé.
\end{proof}

\begin{prop} \label{prop-existence} Pour tout $g \ge 1$, il existe un graphe ordinaire $G$ plongé cellulairement dans $\Sigma_g$, qui possède une MC et qui a $V(G) = 4 \, \left(1 + \dfrac{1}{3g}\right)^{-1}$.
    \end{prop} 
\begin{proof}
Considérons le monographe $G'$ à $3g$ lacets construit lors de la \cref{Prop-HomGenusOptimal} et décrit explicitement suite à celle-ci. Dans le polygone $P$, nous avons ainsi $g$ triangles $T_1, \ldots, T_g$, chacun délimité par les côtés $a_j, b_j$ et $d_j$. Plongeons dans l'intérieur de chaque $T_j$ une copie $T'_j$ du triangle $T_j$, obtenant ainsi trois arêtes $a'_j, b'_j$ et $d'_j$, puis relions par une arête chacun des trois sommets de $T'_j$ au sommet le plus proche de $T_j$. Nous définissons le graphe $G$ comme étant l'image par $q$ de l'union de tous les $T'_j$ et de toutes les arêtes reliant les $T'_j$ aux $T_j$. (Voir la \cref{fig-Prop-7.5-Fig}).

Le graphe $G$ ainsi obtenu est pour ainsi dire une perturbation de $G'$ et correspond à l'éclatement du sommet de $G'$ en une étoile à $3g$ pointes. Les pointes de cette étoile relient le sommet «\,central\,» à chacun des $3g$ autres sommets de $G$ que sont les images par $q$ des coins des $T'_j$. $G$ est donc clairement ordinaire, et il est cellulaire et il a une MC puisqu'il est un éclatement du sommet de $G'$. Puisque $G$ a un sommet de valence $3g$ et $3g$ sommets de valence $3$, nous calculons $V(G) = 4 \, \left(1 + \dfrac{1}{3g}\right)^{-1}$.
\end{proof}

\begin{figure}[H]
  \centering
  
\includegraphics[width=0.7\textwidth]{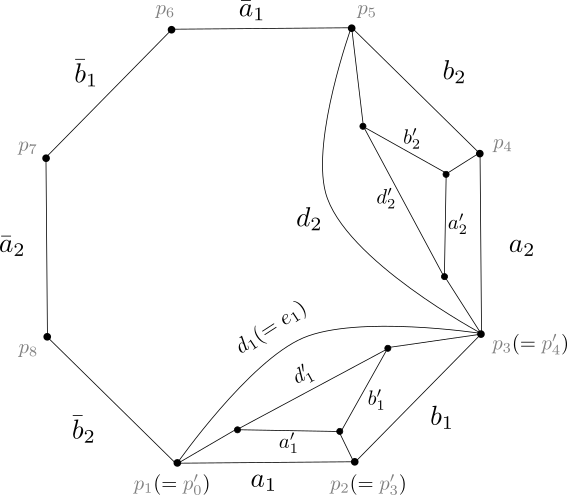}
\caption{Une illustration de la construction de $G$ dans la preuve de \cref{prop-existence} dans le cas où $g=2$, $j=2$}\label{fig-Prop-7.5-Fig}
\end{figure}

Ainsi, pour $g \ge 2$, nous connaissons un graphe ordinaire plongé cellulairement dans $\Sigma_g$ avec une MC tel que $V(G) > 3$. La \cref{fig-optimal1_tore} plus loin montre qu'il existe aussi un tel graphe pour $g=1$.

Rappelons que nous définissons la fonction
\begin{align*}
b(g) &:= 1 + \sqrt{1+6g}
\end{align*}

    \begin{prop} \label{Prop-b(g)} Pour tout $g \ge 1$, $V_c(g) \le b(g)$.
    \end{prop} 
\begin{proof}
Soit $G$ un graphe généralisé plongé cellulairement dans $\Sigma_g$ et qui possède une MC\,; il suffit de montrer que $V := V(G) \le b(g)$. Compte tenu du commentaire précédant l'actuelle Proposition, nous pouvons supposer $V_r > 3$. Ainsi, du fait que $V_r-3 > 0$, ($\diamondsuit_r$) et ($\clubsuit_r$) impliquent
\[   6\gamma - 3 \ge     \left( V_r + 1 \right) \left( V_r - 3 \right) \; \Leftrightarrow \; 0 \ge V_r^2 - 2V_r - 6 \gamma \, . \]
Puisque $b(\gamma)$ est la plus grande racine de ce dernier polynome quadratique, nous obtenons $V_r \le b(\gamma)$. Puisque $G$ est plongé cellulairement,  $\gamma = g$.
\end{proof}

    \begin{prop}\label{Prop-Vc-realise}  Pour tout $g \ge 1$, le supremum $V_c(g)$ est réalisé.
    \end{prop} 
\begin{proof}
Soit $3 < V_* < V_c(g)$ et considérons l'ensemble $X$ des graphes (généralisés, cellulaires, avec MC) $G$ qui satisfont $V_r(G) \ge V_*$. Il résulte alors de ($\diamondsuit_r$) que $S(G) \le S_* := (6g - 3)/(V_* - 3)$. Par cette inégalité, la \cref{Prop-b(g)} et ($\spadesuit_r$), il résulte que $A_r(G) \le A_* := S_* b(g)/2$. Ainsi, l'image de l'application $X \to \mathbb{N} \times \mathbb{N} : G \mapsto (S(G), A_r(G))$ a cardinalité finie, et donc aussi l'image de l'application $X \to \mathbb{Q} : G \mapsto V_r(G)$. Ainsi, $V_c(g)$ est réalisé parmi les éléments de $X$.
\end{proof}

Pour le prochain résultat, rappelons que $\lceil x \rceil_0$ et $\lceil x \rceil_1$ désignent respectivement les plus petits entiers pair et impair plus grands ou égaux à $x$ et que $\lfloor x \rfloor_0$ et $\lfloor x \rfloor_1$ désignent respectivement les plus petits entiers pair et impair plus petits ou égaux à $x$. Rappelons que nous définissons
\begin{align*}
b_r(g) &:= \mathrm{max} \, \left\{  \, \lfloor{S_0(g)} \rfloor_0 - 1 \, , \, \lfloor{S_1(g)} \rfloor_1 - 1\, , \, 3 + \dfrac{6g-4}{\lceil S_0(g) \rceil_0} \, , \,  3 + \dfrac{6g-3}{\lceil S_1(g) \rceil_1} \right\},
\end{align*}
où $S_j(g) := 2 + \sqrt{6g + j}$ pour $j = 0, \, 1$.

 \begin{prop} \label{Prop-br(g)} Pour tout $g \ge 1$, $V_c(g) \le b_r(g) \le b(g) $ 
    \end{prop} 
\begin{proof}
Il suffit de considérer les graphes dans l'ensemble $X$ des graphes généralisés, cellulaires dans $\Sigma_g$, avec MC et satisfaisant $V_r(G) > 3$. Posons $N(S) := 6g-4 + \pi(S)$ et $N(G) = N(S(G))$. Les inégalités ($\clubsuit_r$) et ($\diamondsuit_r$) impliquent $V_r(G) \le \mathrm{min} \, \{ S(G)-1 \, , \, 3 + N(G)/S(G) \}$. Cette inégalité et la \cref{Prop-Vc-realise} nous permettent de déduire que
\begin{align}
\notag  V_c(g) &= \mathrm{max}_{G \in X} \,   V_r(G) \\
\notag &\le  \mathrm{max}_{G \in X} \,  \mathrm{min} \, \{ S(G)-1 \, , \, 3 + N(G)/S(G) \} \\
\notag &= \mathrm{max}_{S} \,  \mathrm{min} \, \{ S-1 \, , \, 3 + N(S)/S \} \, .  
\end{align}

Pour $j = 0,1$, posons $N_j := 6g-4 + j$. Observons que sur l'ensemble des $S$ satisfaisant $\pi(S) = j$, les fonctions $S-1$ et $3 + N_j/S$ sont respectivement croissante et décroissante et coïncident pour $S_j := 2 + \sqrt{N_j+4} = 2 + \sqrt{6g + j}$. Ainsi, 
\[  \mathrm{min} \, \left\{ S-1 \, , \, 3 + \dfrac{N(S)}{S} \right\} = \begin{cases} S-1 &\mbox{ si $S$ satisfait $\pi(S) = j$ et $S \le \lfloor S_j \rfloor_j$}, \\ 3 + \dfrac{N(S)}{S} & \mbox{ si $S$ satisfait $\pi(S) = j$ et $S \ge \lceil S_j \rceil_j$ \, .} \end{cases} \]
En maximisant cette quantité sur les $S$, nous obtenons $b_r(g)$. Il découle de l'argumentaire que $b_r(g) \le S_1(g) - 1 = b(g)$.
\end{proof}


\subsection{\Cref{Lem_optimal}}\label{subsec-Lem_optimal} Nous divisons la démonstration en plusieurs propositions, une pour chaque genre $1 \le g \le 5$. Nous exhiberons un graphe optimal pour chacun de ces $g$, mais différentes représentations de ces graphes seront employées.

\begin{prop}
Soit $G$ un graphe qui est plongé dans une surface de genre $1$ et qui admet une marche à gauche complète. Alors $V(G) \leq 3 \frac{3}{7}$. De plus, il existe un graphe qui réalise l'égalité.
\end{prop}
\begin{proof}
Puisque $V(1) = \mathrm{max} \{ V_c(0), V_c(1)\}$ et $V_c(0) = 3$, il suffit d'établir $V_c(1) = 3 \frac{3}{7}$. Nous savons que $V_c(1) \le b_r(1) = \mathrm{max} \, \left\{ 2 \, , \, 3 \, , \, 3 \frac{1}{3} \, , \, 3 \frac{3}{5} \right\} = 3 \frac{3}{5}$. Nous affirmons que $V_c(1) < b_r(1)$, ce qui exige que nous investiguions davantage la structure de graphes maximaux hypothétiques.

Soit $G$ satisfaisant $V_r(G) \ge 3 \frac{3}{7}$. En combinant cet ansatz à ($\clubsuit_r$) et à ($\diamondsuit_r$) et du fait que $S$ est entier, nous obtenons $5 \le S(G) \le 7$. La possibilité $S=6$ est exclue, puisque l'inégalité  ($\diamondsuit_r$) impliquerait $V_r(G) \le 3 \frac{1}{3}$, à l'encontre de l'ansatz.  

Montrons qu'aucun graphe avec $S(G)=5$ ne respecte l'ansatz. Supposons au contraire qu'un tel graphe $G$ existe. L'inégalité ($\diamondsuit$), qui porte sur la valence totale, donne $V(G) \le 3 \frac{3}{5}$. Ceci et l'ansatz impliquent donc, via ($\spadesuit$), qu'un graphe $G$ avec $S=5$ sommets a précisément $A = 9$ arêtes. Donc $V(G) = 3 \frac{3}{5}$ et l'inégalité ($\diamondsuit$) est saturée. Par la formule d'Euler, nous voyons que $G$ a trois MNC de longueur $3$ et une MC qui a longueur $9$ et qui est donc un cycle eulérien. Ainsi $G$ est construit à partir de $3$ triangles disjoints $T_1$, $T_2$ et $T_3$ via certains recollements de leurs sommets. Puisque les $3$ triangles ont $9$ coins et $S(G)=5$, le recollement n'est pas trivial\,: nous pouvons supposer que $T_1$ et $T_2$ ont des coins identifiés. Si deux points de $T_1$ sont identifiés entre eux ou s'il y a deux paires de coins identifiés $(p,q), (p', q') \in T_1 \times T_2$, alors $G$ n'est pas un graphe ordinaire (il possède au moins une boucle ou une arête multiple)\,; nous avons alors forcément $A_r(G) \le 8$ et donc $V_r(G) \le 3 \frac{1}{5}$, ce qui va à l'encontre de l'ansatz. Donc, après recollement, les triangles $T_1$ et $T_2$ ne partagent qu'un coin en commun et contribuent donc précisément $5$ sommets à $G$, c'est-à-dire tous les sommets de $G$. Les trois sommets de $T_3$ sont donc forcément identifiés à des sommets de $T_1$ ou de $T_2$\,; encore une fois, ceci implique que $G$ n'est pas ordinaire et donc $V_r(G) \le 3 \frac{1}{5}$, à l'encontre de l'ansatz.

Ainsi, l'ansatz ne peut être respecté que par un graphe à $S=7$ sommets. Un raisonnement analogue à celui ci-dessus établit alors que $V(G) = 3 \frac{3}{7}$, que ($\diamondsuit$) est saturée et donc que $G$ est un quotient de $4$ triangles disjoints. La \cref{fig-optimal1_tore} exhibe un tel graphe ordinaire $G$. La \cref{fig-optimal2_tore} illustre une série d'opérations qui permettent d'obtenir ce graphe à partir du monographe homotopique optimal standard dans le tore\,: trois éclatements élémentaires, dédoublement de ces trois arêtes, puis subdivision d'une arête dans chaque paire obtenue.
\end{proof}
\begin{figure}[H]
\includegraphics[width=0.5\textwidth]{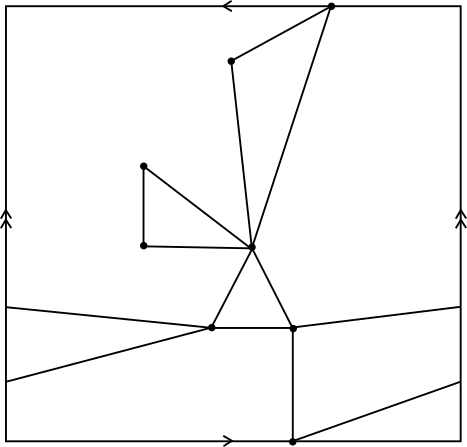}
\caption{Graphe dans le tore réalisant $V(G) = 3 \frac{3}{7}$.}
\label{fig-optimal1_tore}
\end{figure}
\begin{figure}[H]
\includegraphics[angle=270, width=\textwidth]{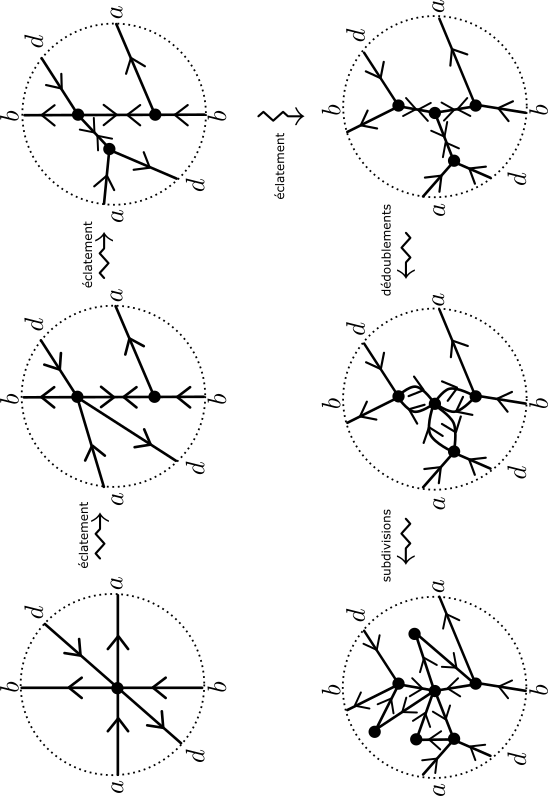}
\caption{Suite d'opérations locales produisant le graphe précédent à partir du monographe optimal standard dans le tore.}
\label{fig-optimal2_tore}
\end{figure}

\begin{prop}
Soit $G$ un graphe qui est plongé dans une surface de genre $2$ et qui admet une marche à gauche complète. Alors $V(G) \leq b_r(2) = 4\frac{1}{3}$ et l'égalité est possible.
\end{prop}
\begin{proof}
Nous avons déjà établi $V_c(0), V_c(1) \le 4 \frac{1}{3}$, tandis que $V_c(2) \le b_r(2) = \mathrm{max} \, \left\{ 2 \, , \, 4 \, , \, 4 \frac{1}{3} \, , \, 4 \frac{2}{7} \right\} = 4 \frac{1}{3}$. Ainsi, $V(2) = \mathrm{max}_{0 \le g \le 2} \, V_c(g) \le 4 \frac{1}{3}$.

Voici une démonstration directe de l'inégalité $V(2) \le 4 \frac{1}{3}$, qui est intéressante du fait qu'elle apporte un certain éclairage sur l'idée générale derrière la \cref{Prop-br(g)}.  Soit $G$ un graphe ordinaire plongé dans $\Sigma_2$ et qui admet une MC. Supposons, par contradiction, que $V_r := V_r(G) > 4 \frac{1}{3}$. Puisque $V \ge V_r > 3$ par hypothèse, les inégalités ($\diamondsuit$) et ($\clubsuit$) donnent
\begin{align*}
S &\leq \frac{6 \gamma -3}{V - 3}.
\end{align*}
En utilisant le fait que $\gamma \leq g$ lorsque $(G,P)$ se plonge dans $\Sigma_g$, nous obtenons ici
\begin{align*}\label{SharpGen2}
\tag{$\ast$} S &\leq \frac{9}{V - 3}. 
\end{align*}
Il résulte des inégalités précédentes que
\begin{align*}
5 \frac{1}{3} < \, &S < 6 \frac{3}{4},
\end{align*}
ce qui implique $S=6$, puisque $S \in \N$. L'inégalité (\ref{SharpGen2}) avec $S=6$ se réarrange alors pour donner
\begin{align*}
V &\leq \frac{27}{6},
\end{align*}
de sorte que $\frac{26}{6} < V \leq \frac{27}{6}$. Puisque $V(G) = \frac{1}{S} \sum_{p \in G} V(p)$, et $\sum_{p \in G} V(p) \in \N$, cela implique que $\sum_{p \in G} V(p) =27$, mais $\sum_{p \in G} V(p) = 2A$ doit être pair, d'où la contradiction recherchée.

Afin d'exhiber un exemple de graphe ordinaire plongé dans une surface de genre $2$ qui sature l'inégalité, nous suivons une stratégie similaire à celle employée dans le cas du genre $1$\,: nous considérons d'abord le monographe homotopique optimal standard $G_0$ pour $g=2$ (\cref{Fig-Genus2Monograph}) et nous explicitons la structure d'incidence de ses arêtes à son sommet (\cref{Fig-Genus2LocalPic}).
  \begin{figure}[H]
  \centering
  \begin{subfigure}{.5\textwidth}
  \centering
\includegraphics[width=\textwidth]{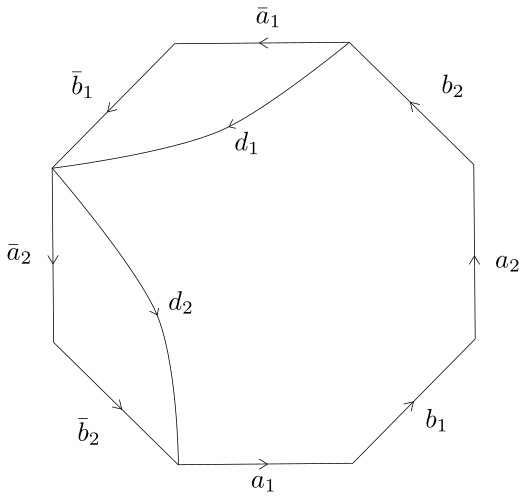}
\caption{La construction du monographe $G_0$}\label{Fig-Genus2Monograph}
\end{subfigure}%
\begin{subfigure}{.5\textwidth}
  \centering
\includegraphics[width=\textwidth]{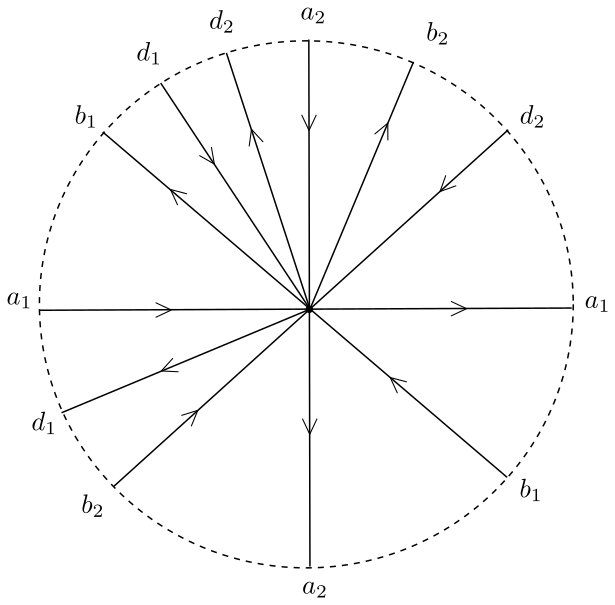}
\caption{Une image locale d'un voisinage du sommet dans  $\Sigma_2$}\label{Fig-Genus2LocalPic}
\end{subfigure}
\caption{}
\label{fig-optimal1_2tore}
\end{figure}
Le graphe ordinaire optimal $G$ est produit en remplaçant $G_0$ dans ce voisinage local par le plongement représenté dans la \cref{Fig-Genus2BlowUp}. Encore une fois, cet éclatement du sommet du monographe s'obtient d'une suite d'éclatements élémentaires (afin de produire les cinq arêtes radiales), puis de l'ajout de deux arêtes parallèles (afin de produire les deux triangles centraux).
\begin{figure}[H]
  \centering
\includegraphics[width=0.65\textwidth]{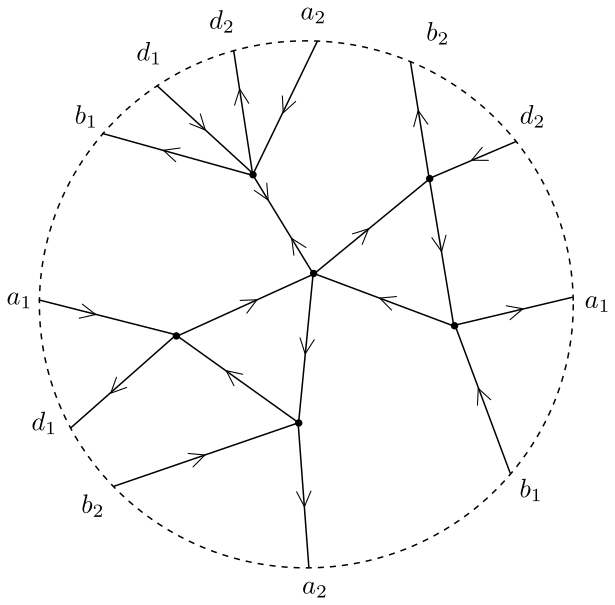}
\caption{L'éclatement du sommet de $G$. La marche à gauche complète sur $G'$ suit les directions indiquées.}\label{Fig-Genus2BlowUp}
\end{figure}
\end{proof}

\begin{prop}
Soit $G$ un graphe qui est plongé dans une surface de genre $3$ et qui admet une marche à gauche complète. Alors $V(G)  \leq b_r(3) = 5\frac{1}{7}$ et l'égalité est possible.
\end{prop}

\begin{proof} Nous avons déjà établi que $V_c(0), V_c(1), V_c(2) \le 5 \frac{1}{7}$, tandis que $V_c(3) \le b_r(3) = \mathrm{max} \, \left\{ 5 \, , \, 4 \, , \, 4 \frac{3}{4} \, , \, 5 \frac{1}{7} \right\} = 5 \frac{1}{7}$. Ainsi, $V(3) = \mathrm{max}_{0 \le g \le 3} \, V_c(g) \le 5 \frac{1}{7}$.

La \cref{fig-optimal_g3} exhibe un graphe qui sature cette borne.  La surface $\Sigma_3$ y est présentée comme quotient d'un polygone à $24$ côtés, à savoir les segments $a_1, \ldots, \bar{a}_{12}$. (À vrai dire, à des fins de lisibilité, la figure présente un rectangle plutôt qu'un polygone à $24$ côtés, mais les côtés verticaux du rectangle sont des signes d'égalité entre les sommets correspondants\,; ces côtés ne font pas partie des $24$ côtés du polygone.) Après identifications de ces $24$ côtés entre eux, nous obtenons une surface orientée de genre $3$ dont la décomposition en CW-complexe comprend sept $0$-cellules, les douze $1$-cellules déterminées par les segments $a_k = \bar{a}_k$ et une $2$-cellule (donnée par l'intérieur du «\,rectangle\,»). Par la formule d'Euler, la surface est donc bien de genre $3$. Un graphe $G_0 \subset \Sigma_3$ à $7$ sommets et $12$ arêtes (les $a_1, \ldots, a_12$) est ainsi obtenu comme $1$-squelette du CW-complexe.

Le graphe $G$ est obtenu en ajoutant les six diagonales $d_1, \ldots, d_6$ à $G_0$. Il est clair que $G$ est plongé cellulairement, qu'il possède six MNC de longueur $3$ (délimitées par les $2$-cellules en demi-lune) et qu'il possède une MC (délimitée par la $2$-cellule du haut complémentaire aux six demi-lunes) qui est un cycle eulérien. $G$ est ordinaire\,: par inspection, aucune arête de $G$ ne lie un sommet à lui-même et aucune paire de sommets n'est liée par deux arêtes distinctes ou plus. Finalement, $V(G) = 2A/S = 36/7$, tel qu'annoncé.
\begin{figure}[H]
  \centering
\includegraphics[width=0.8\textwidth]{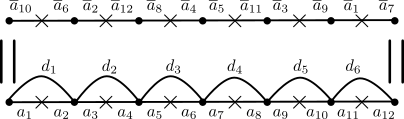}
\caption{Graphe optimal pour $g=3$. Les deux sommets de gauche sont identifiés ensemble, tout comme le sont les deux sommets de droite, d'où les signes d'égalité verticaux.}\label{fig-optimal_g3}
\end{figure}
\end{proof}

\begin{prop}\label{prop-optimal_g4}
Soit $G$ un graphe qui est plongé dans une surface de genre $4$ et qui admet une marche à gauche complète. Alors $V(G) \leq b(4) = 6$ et l'égalité est possible.
\end{prop}

\begin{proof} Nous avons déjà établi que $V_c(0), V_c(1), V_c(2), V_c(3) \le 6$, tandis que $V_c(4) \le b_r(4) = \mathrm{max} \, \left\{ 5 \, , \, 6 \, , \, 5 \frac{1}{2} \, , \, 6\right\} = 6$. Ainsi, $V(4) = \mathrm{max}_{0 \le g \le 4} \, V_c(g) \le 6$. Observons aussi que $b(4) = 6$.

La \cref{fig-optimal_g4} exhibe un graphe qui sature cette borne.  De nouveau, la surface $\Sigma_4$ est présentée comme quotient d'un polygone à $28$ côtés. (Encore une fois, les côtés gauche et droite du rectangle sont des signes d'égalité entre les sommets correspondants et ne font pas partie des $28$ côtés.) Après identifications, ces $28$ côtés déterminent bien une surface orientée de genre $4$, dont la décomposition en CW-complexe comprend sept $0$-cellules, les quatorze $1$-cellules déterminées par les segments $a_k = \bar{a}_k$ et une $2$-cellule (donnée par l'intérieur du «\,rectangle\,»).

Le graphe $G$ est obtenu en ajoutant les sept diagonales $d_1, \ldots, d_7$ aux quatorze arêtes $a_k$ formant le $1$-squelette du CW-complexe. Il est clair que $G$ est plongé cellulairement, qu'il possède sept MNC de longueur $3$ (délimitées par les $2$-cellules en demi-lune) et qu'il possède une MC (délimitée par la $2$-cellule du haut complémentaire aux sept demi-lunes) qui est un cycle eulérien. Encore une fois, nous voyons par inspection que $G$ est ordinaire. Finalement, $V(G) = 2A/S = 6$.
\begin{figure}[H]
  \centering
\includegraphics[width=0.8\textwidth]{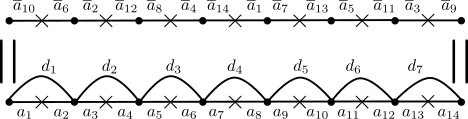}
\caption{Graphe optimal pour $g=4$.}\label{fig-optimal_g4}
\end{figure}
\end{proof}

\begin{prop}
Soit $G$ un graphe  qui est plongé dans une surface de genre $5$ et qui admet une marche à gauche complète. Alors $V(G) \leq b_r(4) = 6 \frac{1}{4}$ et l'égalité est possible.
\end{prop}

\begin{proof} Nous avons déjà établi que $V_c(0), \ldots, V_c(4) \le 6 \frac{1}{4}$, tandis que $V_c(5) \le b_r(5) = \mathrm{max} \, \left\{ 5 \, , \, 6 \, , \, 6 \frac{1}{4} \, , \, 6\right\} = 6$. Ainsi, $V(5) = \mathrm{max}_{0 \le g \le 5} \, V_c(g) \le 6 \frac{1}{4}$.

Un graphe ordinaire $G$ qui sature cette borne est présenté dans la \cref{fig-optimal_g5}. Les sommets de $G$ sont numérotés de $1$ à $8$ et l'ensemble des arêtes incidentes à chaque sommet est dessiné. La polarisation est donnée par les ordres cycliques prescrits par la Figure.

Il est clair que $G$ est ordinaire et il est facile de constater qu'il est connexe. Les sommets $1$ et $2$ ont valence $7$ et les autres sommets ont valence $6$, d'où une valence moyenne $V(G) = 50/8$. Le graphe a donc $25$ arêtes.

La MC, de longueur $26$, est donnée par la suite de sommets suivante\,:
\[ 1, 2, 3, 7, 8, 5, 6, 3, 8, 2, 1, 5, 4, 6, 2, 7, 1, 3, 4, 8, 1, 4, 2, 5, 7, 6, 1. \]
Les huit MNC ont toutes longueur $3$ et sont\,:
\begin{align}
\notag &(i) \, 3, 2, 4, 3.  \; &&(ii) \, 7, 3, 6, 7.  \;  &&&(iii) \, 8, 7, 2, 8. \;  &&&&(iv) \, 5, 8, 4, 5.   \\
\notag &(v) \, 6, 5, 2, 6.  \; &&(vi) \,  8, 3, 1, 8.  \;  &&&(vii) \, 5, 1, 7, 5. \;  &&&&(viii) \, 6, 4, 1, 6.   
\end{align}

\noindent Pour se convaincre que nous avons trouvé toutes les marches à gauche et que la MC est bien complète, il suffit de remarquer que les neuf marches identifiées sont distinctes et parcourent $50$ arêtes orientées (forcément distinctes), et donc qu'elles parcourent toutes les arêtes orientées de $G$. Finalement, le genre du graphe polarisé est bien $\gamma = 5$.
\begin{figure}[H]
  \centering
\includegraphics[angle=270, width=\textwidth]{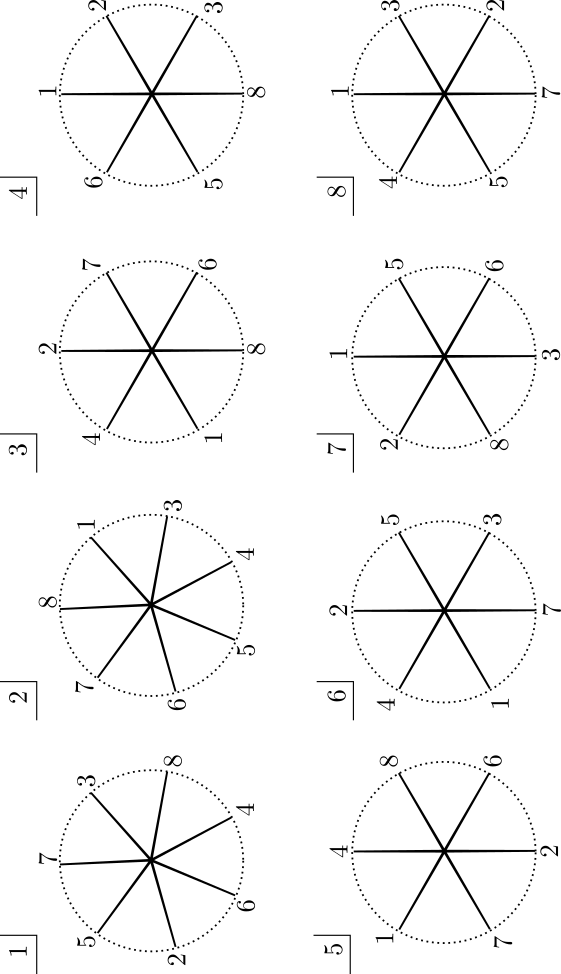}
\caption{Graphe optimal pour $g=5$.}\label{fig-optimal_g5}
\end{figure}
\end{proof}

\begin{figure}[H]
  \centering
\includegraphics[width=0.6\textwidth]{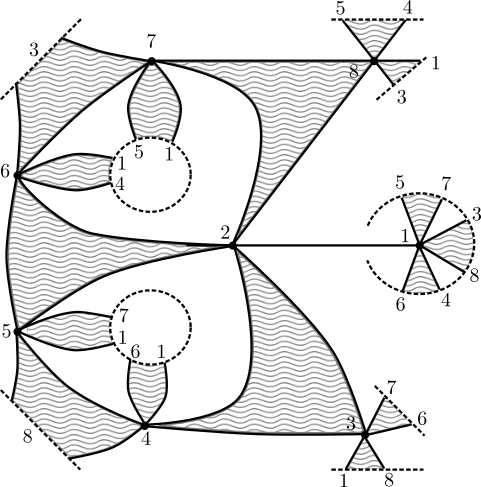}
\caption{Représentation «\,cubiste\,» dans le plan du graphe optimal pour $g=5$. Les huit triangles délimités par les MNC sont hachurés\,: trois sont représentés connexes, les cinq autres sont représentés disconnexes.}\label{fig-optimal_g5_cubiste}
\end{figure}


\section{\Cref{thm-optimalite}} \label{sec-thm-optimalite}

La démonstration du \cref{thm-optimalite} repose sur une compréhension plus précise des structures des graphes $G$ (généralisés, plongés cellulairement dans $\Sigma_g$ et avec MC) qui satisfont $V_r(G) = b_r(g)$, pour les $g$ appropriés. Les parties (a) et (b) sont mutuellement indépendantes, mais elles reposent implicitement toutes deux sur le fait que $b_r(g)$ est alors égal à deux des quatre termes dont $b_r(g)$ est le maximum.

\subsection{Démonstration de la partie (a)} Pour $g = 6k^2$, observons que $S_0(g) := 2 + \sqrt{6g} = 2 + 6k$ est un entier pair, donc $S_0(g) = \lfloor S_0(g) \rfloor_0 = \lceil S_0(g) \rceil_0$. Par ailleurs, quel que soit $g \ge 1$, nous avons $S_0(g) < S_1(g) < S_0(g) + 1$\,; ainsi, pour $g = 6k^2$, nous avons $\lfloor S_1(g) \rfloor_1 = S_0(g)-1$ et  $\lceil S_1(g) \rceil_1 = S_0(g) + 1$. Puisque $6g-4 = S_0(g)(S_0(g) -4) > S_0(g)$, des calculs simples établissent que
\[ \label{eq-astast}\tag{$\ast\ast$} \lfloor S_0(g) \rfloor_0 - 1 = 3 + \dfrac{6g-4}{\lceil S_0(g) \rceil_0} > \mathrm{max} \left\{  \lfloor S_1(g) \rfloor_1 - 1 \, , \, 3 + \dfrac{6g-3}{\lceil S_1(g) \rceil_1} \right\} \, .\]
Bref, nous avons
\[ b_r(g) = \lfloor S_0(g) \rfloor_0 - 1 = 3 + \dfrac{6g-4}{\lceil S_0(g) \rceil_0 } \, .\]

De manière absurde, soit $G$ un graphe généralisé, plongé cellulairement dans $\Sigma_g$ et avec MC qui satisfait $V_r(G) = b_r(g)$. Par le \cref{lem-reduction}, nous pouvons supposer que $G$ satisfait la condition (C). Posons $S = S(G)$. Par l'inégalité ($\clubsuit_r$), nous avons
\[ S_0(g) -1 = b_r(g) = V_r(G) \le S - 1 \, , \]
tandis que par ($\diamondsuit$) et \eqref{eq-astast},
\[  3 + \dfrac{6g-3}{S_0(g)+1} < 3 + \dfrac{6g-4}{S_0(g)} = b_r(g) = V_r(G) \le V(G) \le 3 + \dfrac{6g-3}{S} \, .\]
Il en résulte que $S = S_0(g)$ et que $V_r(G) = V(G) = 3 + (6g-4)/S$. En particulier, $G$ est ordinaire et nous pouvons travailler avec la valence totale.

Puisque $V(G)_v \le S - 1$ pour un graphe ordinaire et du fait que $V(G) = S-1$, nous déduisons que $G$ est un graphe complet sur $S$ sommets et que tous les sommets de $G$ sont de valence impaire.

De $V(G) = 3 + (6g-4)/S$ et de ($\spadesuit$), nous déduisons $A = 3S/2 + 3g - 2$. Du fait que $G$ soit cellulaire dans $\Sigma_g$, la formule d'Euler implique que $G$ possède $f:= F-1 = S/2 + g - 1 = (A-1)/3$ MNC. Par la condition (C) et la définition d'une MC, ces $f$ MNC parcourent (dans une seule direction) entre $A-1$ et $A$ arêtes non orientées distinctes de $G$ (puisque chaque arête non orientée apparaît au plus une fois dans tous les MNC étant donné l'existence d'un MC). Il y a donc au moins $f-1$ MNC de longueur $3$ et la MNC restante a longueur $3$ ou $4$\,; nous considérons ces deux possibilités séparément.

Si la MNC restant a longueur $4$, alors la MC a longueur $A$ et il s'agit donc d'un cycle eulérien. Par le théorème d'Euler--Hierholzer, $G$ n'a donc que des valences paires, ce qui est en contradiction avec le fait que $G$ est le graphe complet sur $S$ sommets.

Si la MNC restant a longueur $3$, alors la MC a longueur $A+1$ et il existe précisément une arête dans $G$ qui soit parcourue dans les deux sens par la MC. Soit $G'$ le sous-graphe obtenu en retirant cette arête. $G'$ est ainsi un graphe ordinaire dont l'ensemble des arêtes est partitionné par les $f$ MNC de longueur $3$ de $G$. Cette observation implique que tous les sommets de $G'$ ont valence paire\,; par conséquent, $G$ n'a que deux sommets de valence impaire. Puisque $S > 2$, il s'agit encore d'une contradiction.

Bref, un tel graphe $G$ n'existe pas, d'où $V_c(g) < b_r(g)$.

\subsection{Démonstration de la partie (b)} Nous procédons par analyse-synthèse\,: nous identifions d'abord diverses propriétés des graphes qui peuvent réaliser la borne $b(g)$, puis nous montrons qu'il existe de tels graphes pour les $S$ annoncés.

\subsubsection{Analyse} Le résultat suivant est notre boussole\,:

\begin{prop}\label{prop-graphecomplet}
Soient $g \geq 1$ et $G \subset \Sigma_g$ un graphe généralisé, plongé cellulairement, avec MC, satisfaisant la condition (C) et tel que $V_r(G) = b(g)$. Alors $G$ est un graphe complet à $S$ sommets, $S \equiv 1 \mbox{ ou } 3 \mbox{ mod } 6$ et $g = (S-1)(S-3)/6$. De plus, la MC de $G$ est un cycle eulérien et les MNC ont toutes longueur $3$.
\end{prop}

\begin{proof}
Par hypothèse et par ($\diamondsuit$), $b(g) = V_r(G) \le V(G) \le 3 + (6g-3)/S$. Ainsi, $1+6g  \le (2 + (6g-3)/S)^2$, et comme $g \geq 1$ ceci implique $S^2 - 4S - (6g-3) \le 0$, c’est-à-dire $S \le 2 + \sqrt{1 + 6g} = 1 + V_r(G)$. Puisque ($\clubsuit_r$) stipule l'inégalité opposée, nous obtenons $S = 1 + V_r(G) = 2 + \sqrt{1+6g}$ et donc aussi $V(G) = V_r(G)$. D’une part, ceci implique que $G$ est un graphe ordinaire et complet. D’autre part, puisque $g$ est entier, $(S-2)^2 = 1 + 6g$ est impair et donc $S$ est impair. De plus, $g = (S-1)(S-3)/6$, donc $S \equiv 1 \mbox{ ou } 3 \mbox{ mod } 6$.

Il s'ensuit aussi que $A(G) = 3(S-1)/2 + 3g$. Par la formule d'Euler, nous trouvons que $G$ possède $f:= F-1 = (S-1)/2 + g = A/3$ MNC. Par la condition (C) et la définition de la MC, nous déduisons que les MNC parcourent (dans un seul sens) précisément $3f = A$ arêtes non orientées distinctes, c'est-à-dire toutes les arêtes de $G$. Ceci montre que les $f$ MNC ont toutes longueur $3$ et que la MC, ayant longueur $A$, est un cycle eulérien.
\end{proof}

L'idée pour démontrer le \cref{thm-optimalite}(b) consiste donc, étant donné un $S \equiv 1 \mbox{ ou } 3 \mbox{ mod } 6$ convenable, à trouver une polarisation du graphe complet $K_S$ qui admette une MC qui soit un cycle eulérien et dont toutes les MNC aient longueur $3$. En effet, le théorème fondamental des plongements cellulaires se chargerait ensuite de plonger $K_S$ cellulairement dans la surface $\Sigma_g$ appropriée.

Il importe de souligner qu'il n'y a aucune obstruction évidente à l'existence d'une telle polarisation, quel que soit $S \equiv 1 \mbox{ ou } 3 \mbox{ mod } 6$. D'un côté, pour $S$ est impair, le théorème d'Euler--Hierholzer assure que $K_S$ admet un cycle eulérien. Il s'avère que tout cycle eulérien permet de définir une polarisation ayant ce cycle pour MC, mais une telle polarisation n'a pas forcément des MNC de longueur $3$. D'un autre côté, c'est un fait classique que $S \equiv 1 \mbox{ ou } 3 \mbox{ mod } 6$ est la condition nécessaire et suffisante pour que $K_S$ admette un système de triples de Steiner, c'est-à-dire une partition de l'ensemble de ses arêtes en $3$-cycles (une preuve simple de ce fait se trouve dans \cite{S2}). Il s'avère que tout système de triples de Steiner sur $K_S$ permet de définir une polarisation ayant ces triples d'arêtes parmi ses MNC, mais une telle polarisation n'a pas forcément de MC.

Notre défi consiste donc à réconcilier ces deux facettes en trouvant un cycle eulérien et un système de triples de Steiner qui soient compatibles en ce sens qu'ils proviennent d'une polarisation ayant le cycle eulérien pour MC et les triples de Steiner pour MNC. Notre stratégie pour y parvenir est une généralisation de la stratégie derrière la construction des graphes des figures \ref{fig-optimal_g3} et \ref{fig-optimal_g4}. (En fait, la \cref{prop-optimal_g4} n'est nulle autre que le cas particulier $S=7$ du \cref{thm-optimalite}(b).) Nous synthétiserons le résultat suivant\,:

\begin{prop}\label{prop-sature}
Soit $G$ un graphe ordinaire avec $S$ impair plongé cellulairement dans $\Sigma_g$, qui a une MC et qui sature la borne ($\diamondsuit$). Alors la MC de $G$ est un cycle eulérien et les MNC ont toutes longueur $3$. De plus, il existe un polygone $\Pi$ à $4f = 4g + 2(S-1)$ côtés ayant les propriétés suivantes\,:
\begin{enumerate}[(i)]
\item Il y a $2f$ côtés indicés $a_1, \ldots, a_{2f}$ et orientés selon le sens antihoraire. Leur position dans $\partial \Pi$ est contrainte comme suit\,: pour chaque $j = 1, \ldots, f$, les côtés $a_{2j-1}$ et $a_{2j}$ sont consécutifs dans $\partial \Pi$ suivant le sens antihoraire.

\item Les $2f$ côtés restants sont indicés $\bar{a}_1, \ldots, \bar{a}_{2f}$ et orientés selon le sens horaire de $\partial \Pi$.

\item Pour chaque $j=1, \dots, f$, il y a un segment plongé $d_j$ joignant dans $\mathrm{int}(\Pi)$ la source de $a_{2j-1}$ et la cible de $a_{2j}$, et orienté de la sorte. Les $d_j$ sont ne peuvent s'intersecter qu'aux coins de $\Pi$.
\end{enumerate}
La surface $\Sigma_g$ est obtenue comme quotient de $\Pi$ en identifiant, pour tout $k=1, \ldots, 2f$, les côtés $a_k$ et $\bar{a}_k$ de façon à ce que les deux côtés induisent une même orientation sur le segment $e_k$ résultant. Le graphe $G \subset \Sigma_g$ est alors donné par l'union $(\bigcup_{1 \le k \le 2f} e_k ) \cup (\bigcup_{1 \le j \le f} d_j )$. Les MNC de $(G,P)$ sont les cycles $[a_{2j-1}, a_{2j}, \bar{d}_j]$ ($j=1, \ldots, f$) et la MC est le bord de la $2$-cellules dans $\Pi$ délimitées par les $d_j$ et les $\bar{a}_k$ parcouru dans le sens antihoraire.
\end{prop}
\begin{proof} Par saturation de ($\diamondsuit$), $A = 3(S-1)/2 + 3g$. Par la formule d'Euler, $G$ possède $f:=F-1 = (S-1)/2 + g = A/3$ MNC, qui parcourent précisément $3f = A$ arêtes. Bref, toutes les MNC ont longueur $3$ et la MC, ayant longueur $A$, est un cycle eulérien.

Chaque arête non orientée de $G$ apparaît précisément dans une seule MNC. Par le théorème fondamental des plongements polarisés  (\cref{thm-ThmFond}), nous pouvons exprimer $\Sigma_g$ sous la forme d'un CW-complexe dont $G$ est le $1$-squelette et qui a $F = f+1$ $2$-cellules collées à $G$ suivant chacune une marche à gauche de $G$. Le reste de la démonstration consiste simplement à décrire ce recollement de ces $2$-cellules le long de $G$ en différentes étapes.

Les $f$ $2$-cellules associées aux MNC être interprétées comme étant des polygones $\Pi_j$ à $3$ côtés ($j=1, \dots, f$), et la $2$-cellule associée à la MC comme étant un polygone $\Pi_0$ à $A = 3f$ côtés. Pour chaque $\Pi_j$ avec $1 \le j \le f$, nous sélectionnons un côté que nous notons $\bar{d}'_j$ (orienté selon le sens antihoraire de $\partial \Pi_j$). Ces $f$ segments correspondent à $f$ arêtes distinctes $d''_j$ dans $G$ et appartiennent à des MNC distinctes\,; ces arêtes sont parcourues en sens inverse par la MC. Ainsi, il y a $f$ côtés de $\Pi_0$ qui correspondent à ces arêtes inversées\,; nous notons ces côtés $d'_j$ et nous les orientons selon le sens antihoraire de $\partial \Pi_0$.

Pour chaque $j=1, \ldots, f$, nous collons $\Pi_j$ le long de $\Pi_0$, par identification du côté $\bar{d}'_j \subset \partial \Pi_j$ et du côté $d'_j \subset \partial \Pi_0$, de sorte que le segment résultant $d_j$ ait l'orientation de $d'_j$ et l'orientation inverse de $\bar{d}'_j$. Nous obtenons ainsi le polygone $\Pi$ à $4f$ côtés.

Pour $j=1, \ldots, f$, les deux côtés restants de $\Pi_j \subset \Pi$ sont nommés (dans l'ordre antihoraire) $a_{2j-1}$ et $a_{2j}$. Tous ces côtés $a_k$ sont en bijection avec les arêtes de $G \setminus \{ d''_1, \ldots, d''_f\}$\,; ces arêtes sont parcourues en sens inverse par la MC. Il y a donc une bijection entre l'ensemble des $a_k$ et l'ensemble des côtés encore non indicés de $\Pi_0$\,; nous notons $a_k \mapsto \bar{a}_k$ cette bijection.

Ainsi, en collant chaque côté $a_k$ de $\Pi$ au côté $\bar{a}_k$ correspondant, nous aboutissons à la même structure de CW-complexe sur $\Sigma_g$.
\end{proof}

\subsubsection{Stratégie de synthèse} Notre stratégie pour démontrer le \cref{thm-optimalite}(b) vise, pour tout $S \equiv 1 \mbox{ ou } 3 \mbox{ mod } 6$ convenable et $g = (S-1)(S-3)/6$, à exprimer explicitement $\Sigma_g$ comme un quotient d'un polygone $\Pi$ à $4f = 4g + 2(S-1)$ côtés et décoré de $f$ diagonales $d_1, \ldots, d_f$ (conformément aux points (i)--(ii)--(iii) ci-dessus) de façon à ce que $K_S \subset \Sigma_g$ soit donné par $G := (\bigcup_{1 \le m \le 2f} e_m ) \cup (\bigcup_{1 \le n \le f} d_n )$.

L'enjeu ici consiste à ordonner les indices $a_1, \ldots, \bar{a}_{2f}$ le long de $\partial \Pi$ de façon à ce que le quotient donne bien la surface $\Sigma_g$ et le graphe $K_S$. Cela revient à construire une \emph{valuation} $v$ appropriée qui associe à chaque coin de $\Pi$ un sommet de $K_S$. Nous détaillons notre stratégie en quelques étapes.

\textbf{Étape 1}\,: Nous faisons l'ansatz suivant\,: les côtés $a_1$ à $a_{2f}$ sont consécutifs, dans le sens antihoraire. Ainsi, le bord $\partial \Pi$ est divisé en deux hémisphères, l'un contenant tous les $a_m$ et l'autre contenant tous les $\bar{a}_m$.

Comme dans les figures \ref{fig-optimal_g3} et \ref{fig-optimal_g4}, il nous sera utile de penser au polygone $\Pi$ comme étant deux lignes horizontales alignées une au-dessus de l'autre pour former un rectangle. Les $2f$ côtés $\bar{a}_m$ de $\Pi$ sont distribués sur la face du haut (suivant un ordre encore à identifier) et les $2f$ côtés $a_1, \ldots, a_{2f+1}$ sont consécutivement distribués, de gauche à droite, sur la face du bas. Les diagonales $d_n$ sont tracées en demi-lunes afin de former $f$ triples d'arêtes  $(a_{2n-1}, a_{2n}, d_n)$. Notons que la face du bas du rectangle a $2f+1$ «\,coins\,», dénotés de gauche à droite $p_1, p_2, \dots, p_{2f}$, de sorte que $a_{2n-1} = p_{2n-1}p_{2n}$, $a_{2n} = p_{2n} p_{2n+1}$ et $d_n = p_{2n-1}p_{2n+1}$. Nous étiquetons de gauche à droite $p'_1, \dots, p'_{2f+1}$ les «\,coins\,» de la face du haut du rectangle et nous posons $a'_m := p'_m p'_{m+1}$ ($m = 1, \ldots, 2f$). Les côtés verticaux gauche et droite du rectangle sont des artéfacts de la présentation rectangulaire\,; nous y inscrivons des signes d'égalité pour souligner que les extrémités de gauche $p_1$ et $p'_1$ correspondent à un même coin de $\Pi$ et similairement pour les extrémités de droite $p_{2f+1}$ et $p'_{2f+1}$.

\textbf{Étape 2}\,: Nous choisissons un système de triples de Steiner sur l'ensemble des $S$ sommets de $K_S$\,; observons qu'un tel système consiste en $A/3 = f$ triples de sommets. Nous associons ensuite à chacun des coins $p_m$ du rectangle un sommet $v(p_m) \in K_S$, de façon à ce que les $f$ triples $(v(p_{2n-1}), v(p_{2n}), v(p_{2n+1}))$ ($j = n, \ldots, f)$ correspondent, dans un certain ordre, aux $f$ triples du système de Steiner choisi. Dans la pratique, nous imposerons toujours $v(p_1) = v(p_{2f+1})$.

Cette valuation $p_m \mapsto v(p_m)$ détermine une application de l'ensemble des côtés (orientés) $a_m$ et des diagonales (orientées) $d_n$ vers l'ensemble des arêtes orientées de $K_S$, à savoir $(p_m,p_{m'}) \mapsto v(p_m ,p_{m'}) = (v(p_m) , v(p_{m'})) \in \mathcal{A}_{orient}(K_S)$. Par définition d'un système de Steiner, l'application obtenue en oubliant l'orientation, $(p_m,p_{m'}) \mapsto v(p_m ,p_{m'}) = \{ v(p_m) , v(p_{m'})\} \in \mathcal{A}(K_S)$, s'avère injective.

\textbf{Étape 3}\,: Nous associons à chacun des coins $p'_m$ un sommet $v(p'_m) \in K_S$. Ceci induit une application $a'_m = (p'_m, p'_{m+1}) \mapsto v(a'_m) = (v(p'_m), v(p'_{m+1})) \in \mathcal{A}_{orient}(K_S)$. L'association $v$ doit respecter diverses contraintes\,:
\begin{enumerate}
\item Puisque $p_1$ et $p'_1$ incarnent le même coin de $\Pi$, nous exigeons $v(p_1) = v(p'_1) \in K_S$. Similairement, $v(p_{2f+1}) = v(p'_{2f+1}) \in K_S$. (Comme nous le mentionnons, en pratique, le même sommet de $K_S$ sera associé à ces quatre coins.)

\item Nous exigeons que les ensembles $\{v(a_m)\}_{1 \le m \le 2f}$ et $\{v(a'_m)\}_{1 \le m \le 2f}$ soient identiques dans $\mathcal{A}_{orient}(K_S)$. Comme la valuation est injective sur les arêtes $a_1,\ldots,a_{2f}$, il existe une bijection $\phi: \lbrace 1, \ldots, 2f \rbrace \circlearrowleft$ telle que $v(a_m) = v(a'_{\phi(m)})$. Nous posons alors $\bar{a}_m := a'_{\phi(m)}$. (Cette condition est requise par la \cref{prop-sature}.)

\item Nous souhaitons que la bijection $\phi$ soit suffisamment mélangeante -- en un sens que nous précisons à l'Étape 4 -- pour que le polygone $\Pi$ donne bien la surface $\Sigma_g$ et le graphe $K_S$ après recollement. En particulier, nous exigeons $v(a'_1) \neq v(a_1)$, $v(a'_{2f}) \neq v(a_{2f})$ et $\phi(m+1) \neq \phi(m) + 1$ pour tout $1 \le m \le 2f-1$, c'est-à-dire qu'aucune séquence de côtés consécutifs du type $(\bar{a}_m, \bar{a}_{m+1})$ n'apparaît en haut du rectangle. (Ces dernières conditions sont nécessaires pour obtenir le graphe $K_S$ après recollement\,; sans elles, le graphe obtenu aurait des sommets de valence $2$ ou $4$ et il ne s'agirait donc pas de $K_S$ où $S \equiv 1 \mbox{ mod } 6$.)
\end{enumerate}

\textbf{Étape 4}\,: Le résultat de la dernière étape est une séquence $(a'_1, \ldots, a'_{2f}) = (\bar{a}_{\phi^{-1}(1)}, \ldots, \bar{a}_{\phi^{-1}(2f)})$, encodée simplement sous la forme $(\phi^{-1}(1), \ldots, \phi^{-1}(2f))$. Cette séquence détermine un certain quotient $\Sigma$ de $\Pi$ en identifiant les côtés $a_m$ aux côtés $\bar{a}_m$, ainsi qu'un certain graphe $G \subset \Sigma$ ayant pour arêtes les images des $a_m$ et des $d_n$. Il ne nous reste plus qu'à vérifier si $\Sigma = \Sigma_g$ et si $G = K_S$. Pour ce faire, il suffit de vérifier si les coins de $\Pi$ donnent lieu à précisément $S$ sommets distincts dans le quotient.

En effet, d'une part, la surface $\Sigma$ est compacte et orientée. Soit $S'$ le nombre de points que les coins de $\Pi$ définissent dans $\Sigma$\,; soient $e_m$ ($1 \le m \le 2f$) les images dans $\Sigma$ des segments $a_m$. La surface $\Sigma$ admet ainsi une structure de CW-complexe composée d'une $2$-cellule (l'intérieur de $\Pi$), de $2f$ $1$-cellules (les $e_m$) et de $S'$ $0$-cellules (les images des coins de $\Pi$). Comme $f=(S-1)/2 +g$, par la formule d'Euler, il en résulte que $\Sigma = \Sigma_g$ si et seulement si $S' = S$.

D'autre part, soient $q$ et $q'$ deux coins de $\Pi$ qui sont identifiés dans le quotient. Supposons que $q$ débute (pour une certaine orientation) un côté $c_1$ de $\Pi$. Alors il existe une chaîne de coins $q_1 = q, q_2, \dots, q_n = q'$ où $q_2$ débute $\bar{c}_1$ et débute (pour une certaine orientation) un autre côté $c_2$ de $\Pi$, $q_3$ débute le côté $\bar{c}_2$ et débute (pour une certaine orientation) un côté $c_3$, etc. Par la condition (2) de l'Étape 3, si $q$ et $q'$ sont liés par une telle chaîne, alors $v(q) = v(q')$. Donc $S' \ge S$ et $K_S$ est le quotient de $G$ obtenu en identifiant les sommets de $G$ de même valuation $v$. Il en résulte que $G = K_S$ si et seulement si $S' = S$. 

Le critère $S' = S$ permet d'éclaircir la signification de la condition (3) de l'Étape 3\,: il faut que la bijection $\phi$ soit suffisamment mélangeante pour que $v(q) = v(q')$ seulement si les coins $q$ et $q'$ sont identifiés dans le quotient.

\subsection{Synthèse} Pour $S=9 \equiv 3 \mbox{ mod } 6$, donc pour $g = 8$ et $f = 12$, en employant le système de Steiner décrit dans \cite[$\S 3$]{S2}, nous avons trouvé la solution $(\phi^{-1}(k))_{1 \le k \le 24}$ suivante\,:
\[  (19, 11, 23, 6, 21, 15, 20, 7, 13, 10, 16, 5, 14, 2, 9, 24, 1, 22, 4, 17, 8, 3, 12, 18) \, . \]
Nous laissons au lecteur le soin de confirmer qu'il s'agit bien d'une solution, c'est-à-dire de vérifier que $S' = S = 9$.

Considérons maintenant $S = 6k+1$ premier où $k = 2l+1$ est impair, bref $S = 12l+7$ ($l \ge 0$). Nous pouvons alors construire des solutions en raffinant la stratégie précédente\,: nos hypothèses sur $S$ et les solutions des articles \cite{S1, O} nous serviront à former des systèmes de Steiner (décrits dans \cite[$\S 2$]{S2}) qui faciliteront l'obtention de valuations $v$ appropriées.

Nous aurons besoin des deux lemmes suivants\,:

\begin{lem}[\cite{S2}]\label{lem-Skolem}
Pour tout entier $k \ge 1$, il existe $3k$ entiers $\{\alpha_j, \beta_j, \gamma_j\}_{1 \le j \le k} \subset \mathbb{Z}$ tels que\,:
\begin{itemize}
\item[$\bullet$] Les $3k$ éléments sont distincts, non nuls et compris entre $1$ et $3k+1$.

\item[$\bullet$] Aucune paire de ces éléments ne somme à $6k+1$. Autrement dit, un seul des deux entiers $3k$ et $3k+1$ apparaît parmi les $3k$ éléments.

\item[$\bullet$] $j = \alpha_j < \beta_j < \gamma_j$ et $\alpha_j + \beta_j = \gamma_j$ pour tout $1 \le j \le k$.
\end{itemize}
\end{lem}

\begin{proof}
Posons $\alpha_j = j$, $\beta'_j = \beta_j - k$ et $\gamma'_j = \gamma_j - k$. Le problème consiste alors à trouver $2k$ entiers distincts $\{\beta'_j, \gamma'_j\}_{1 \le j \le k}$ entre $1$ et $2k+1$ tels que $\gamma'_j - \beta'_j = j$ pour tout $1 \le j \le k$ et tels qu'un seul des entiers $2k$ et $2k+1$ apparaisse parmi ces $2k$ entiers. Skolem \cite{S1} a montré que de tels entiers existent, aucun égal à $2k+1$, si et seulement si $k \equiv 0 \mbox{ ou } 1 \mbox{ modulo } 4$\,; O'Keefe \cite{O} a montré que de tels entiers existent, aucun égal à $2k$, si et seulement si $k \equiv 2 \mbox{ ou } 3 \mbox{ modulo } 4$.
\end{proof}

\begin{lem}\label{lem-Skolem_copremier}
Soient $\{\alpha_j, \beta_j, \gamma_j\}_{1 \le j \le k} \subset \mathbb{Z}$ comme dans le \cref{lem-Skolem}. Supposons que $S := 6k+1$ soit premier. Alors les sommes $\alpha := \sum_{j=1}^k \alpha_j$, $\beta := \sum_{j=1}^k \beta_j$ et $\gamma := \sum_{j=1}^k \gamma_j$ sont copremiers avec $S$.
\end{lem}

\begin{proof}
Posons $r = 0$ si si $k \equiv 0 \mbox{ ou } 1 \mbox{ modulo } 4$ et $r=1$ sinon. Puisque $S$ est premier, il suffit de montrer que $S$ ne divise aucune des sommes. Il est clair que le nombre premier $S$ ne divise pas $\alpha = k(k+1)/2$. Ensuite, observons que 
\[ 2 \gamma = \alpha + \beta + \gamma = r + \dfrac{3k(3k+1)}{2} = r + \dfrac{(S-1)(S+1)}{8} \, .\]
Il est clair que $S$ ne divise pas $\gamma$ si $r=0$. Si $r=1$, alors $S$ divise $\gamma$ seulement si $S$ divise $S^2 + 7$, donc seulement si $S=7$, ce qui oblige $k = 1$ et donc $r=0$\,; bref, $S$ ne divise pas $\gamma$ si $r=1$. Finalement, nous avons
\[ \beta = \gamma - \alpha= \dfrac{7S^2 - 8S + 1 + 72r}{144}\, .\]
Il est clair que $S$ ne divise pas $\beta$ si $r=0$. Si $r=1$, alors $S$ divise $\beta$ seulement si $S = 73$, ce qui oblige $k = 12$ et donc $r = 0$\,; bref, $S$ ne divise pas $\beta$ si $r=1$.
\end{proof}

\textbf{Étape 1}\,: Soit $S = 6k+1$ ($k \ge 1$), de sorte que $g = k(6k-2)$ et $f = k(6k+1) = kS$. Notre rectangle $\Pi$ de longueur $2f$ est donc la concaténation de $k$ sous-rectangles $\Pi_j$ ($j=1, \dots, k$) de longueur $2S$ chacun, où le sous-rectangle $\Pi_j$ est formé par les segments horizontaux $(p_{2S(j-1) + 1}, \dots, p_{2Sj+1})$ et $(p'_{2S(j-1) + 1}, \dots, p'_{2Sj+1})$ et est décoré des arcs $d_{S(j-1) + 1}, \ldots, d_{Sj}$.

Pour la suite des choses, fixons un étiquettage $0, \dots, S-1$ des sommets de $K_S$. Il nous sera utile de penser à ces étiquettes comme étant les éléments de $\mathbb{Z}/S\mathbb{Z}$.

\textbf{Étape 2}\,: Nous allons construire le système de Steiner sur $K_S$ et la valuation $v$ sur les $p_1, \dots, p_{2f+1}$ simultanément.  Tel que montré dans \cite[$\S 2$]{S2}, le \cref{lem-Skolem} permet de construire un système de Steiner sur $K_S$ quel que soit $S = 6k+1$. Ainsi, le système de triples de Steiner que nous nous apprêtons à construire quand $S$ est premier n'est qu'un cas particulier de la construction donnée par Skolem, mais il a le mérite d'admettre une structure particulièrement régulière, ce qui nous aidera à accomplir les autres Étapes de notre stratégie.

Pour $1 \le j \le k$, considérons le rectangle $\Pi_j$, dont les «\,coins\,» de la face inférieure sont $p_{2S(j-1) + 1}, \dots, p_{2Sj+1}$. Considérons aussi le triple $(\alpha_j, \beta_j, \gamma_j)$ donné par le \cref{lem-Skolem}, que nous interprétons comme sous-ensemble de $\mathbb{Z}/S\mathbb{Z}$. Nous définissons une valuation $v_j$ de ces coins dans $\mathcal{S}(K_S) =  \mathbb{Z}/S\mathbb{Z}$ par les deux règles suivantes\,:
\begin{enumerate}[(i)]
\item $v_j(p_{2S(j-1)+1}) = 0$.
\item $v_j(p_{m+1}) = v_j(p_m) + \alpha_j$ si $m$ est impair et $v_j(p_{m+1}) = v_j(p_m) + \beta_j$ si $m$ est pair.
\end{enumerate}

\noindent Ces règles impliquent $v_j(p_{m+2}) = v_j(p_{m}) + \gamma_j$ pour tout $2S(j-1) + 1 \le m \le 2Sj - 1$. En particulier, $v_j(p_{2Sj + 1}) = v_j(p_{2S(j-1) + 1}) + S\gamma_j = 0 = v_{j+1}(p_{2Sj+1})$. Ceci prouve que les divers $v_j$ déterminent ensemble une valuation $v$ sur la face inférieure de $\Pi$.

Pour chaque $1 \le j \le k$, du fait que $\gamma_j$ est copremier avec $S$, il se trouve que les $S$ valeurs $v_j(p_{2S(j-1)+ 2m - 1}) \in K_S$ ($1 \le m \le S$) sont distinctes\,; elles énumèrent donc tous les sommets de $K_S$ sans répétition. Il en va de même des sommets  $v_j(p_{2S(j-1) + 2m})$ ($1 \le m \le S$), puisque cet ensemble n'est qu'une translation de l'ensemble précédent. Donc chaque sommet $v \in K_S$ apparait précisément deux fois parmi les $p_{2S(j-1) + m}$ ($1 \le m \le 2S$), en fait pour deux $m$ de parités différentes.

À ce point-ci et en prévision des Étapes suivantes, il convient de dire qu'un côté $a_m$ de $\Pi_j$ est «\,un côté $\alpha_j$\,» si $m$ est impair et est «\,un côté $\beta_j$\,» si $m$ est pair\,; cette terminologie reflète simplement la différence de valuations entre les extrémités droite et gauche de $a_m$. Le paragraphe précédent implique donc que pour chaque sommet $v \in K_S$ et chaque $1 \le j \le k$, il y a parmi les côtés $\alpha_j$ un seul côté $\alpha_j(v)$ qui débute par $v$ et un seul (autre) côté $\alpha'_j(v)$ qui termine par $v$\,; similairement, parmi les côtés $\beta_j$, il y a un seul côté $\beta_j(v)$ qui débute par $v$ et un seul (autre) côté $\beta'_j(v)$ qui termine par $v$.  Conséquemment, pour tout $v \in K_S$ et tout $1 \le j \le k$, parmi les diagonales $d_{S(j-1) + n}$ ($1 \le n \le S$), il y en a une seule qui débute par $v$ et il y en a une seule autre qui termine par $v$.

Nous affirmons que l'ensemble $\{ (v(p_{2m-1}), v(p_{2m}), v(p_{2m+1})  \}_{1 \le m \le f}$, c'est-à-dire l'ensemble des triplets de valuation des coins des $f$ demi-lunes de $\Pi$, est un système de triples de Steiner de $K_S$. Cela découle du paragraphe précédent et de \cite[$\S 2$]{S2}. Alternativement, il suffit de vérifier que la valuation induite sur les côtés $a_m$ et les diagonales $d_n$ à valeurs dans $\mathcal{A}(K_S)$ est injective. Soient $(p_{m}, p_{m'})$ et $(p_{l}, p_{l'})$ ($m'-m, l'-l \in \{1,2\}$) deux segments distincts ayant la même valuation, c'est-à-dire que $\{ v(p_m), v(p_{m'}) \} = \{ v(p_l), v(p_{l'})\} \in \mathcal{A}(K_S)$.
\begin{itemize}
\item[$\bullet$] La possibilité $v(p_m) = v(p_{l'})$ et $v(p_{m'}) = v(p_l)$ est exclue. Autrement, nous aurions $v(p_{m'}) - v(p_m) = -(v(p_{l'}) - v(p_{l}))$ et $v(p_{m'}) - v(p_m), v(p_{l'}) - v(p_l) \in \{ \alpha_j, \beta_j, \gamma_j \}_{1 \le j \le k} \subset \mathbb{Z}/S\mathbb{Z}$, de sorte que deux des $\alpha_1, \ldots, \gamma_k$ aurait une somme nulle dans $\mathbb{Z}/S\mathbb{Z}$, ce qui est absurde.

\item[$\bullet$] Donc $v(p_m) = v(p_l)$ et $v(p_{m'}) = v(p_{l'})$. Ainsi $v(p_{m'}) - v(p_m) = v(p_{l'}) - v(p_l) \in \{ \alpha_j, \beta_j, \gamma_j \}_{1 \le j \le k}$. Donc les segments $p_m p_{m'}$ et $p_l p_{l'}$ appartiennent au même rectangle $\Pi_j$, sont tous les deux des $a_m$ et des $d_l$ et sont ainsi en fait égaux.
\end{itemize}

 \begin{figure}[h]
  \centering
\includegraphics[angle=270, width=1\textwidth]{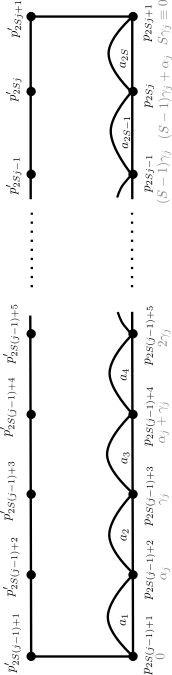}
\caption{Illustration du rectangle $\Pi_j$ apparaissant dans l'étape 2 de la Synthèse. La valeur de la valuation sur les «\,coins\,» est indiquée en gris.}\label{fig-SyntheseDrawing}
\end{figure}

\textbf{Étape 3}\,: Nous définissons la valuation $v$ sur les $p'_{m}$ ($1 \le m \le 2f + 1$) à valeurs dans $\mathcal{S}(K_S)$ comme suit. D'abord, posons $v_1 := v(p'_1) = 0$. Ensuite, pour $m \ge 2$, nous définissons $v_m := v(p'_m)$ par récurrence comme suit (rappelons que $f = kS$)\,:
\begin{enumerate}[(i)]
\item Si $1 \le m \le f $ et $m \equiv j \mbox{ mod } k$, alors $v_{m+1} := v_{m} + \beta_j$. Autrement dit, $(v_{m}, v_{m+1}) := v(\beta_j(v_{m}))$, d'où $a'_{m} = \overline{\beta_j(v_{m})}$.

\item Si $f+1 \le m \le 2f$ et $m \equiv j \mbox{ mod } k$, alors $v_{m+1} := v_{m} + \alpha_j$. Autrement dit, $(v_{m}, v_{m+1}) := v(\alpha_j(v_{m}))$, d'où $a'_{m} = \overline{\alpha_j(v_{m})}$.
\end{enumerate}

\noindent Concrètement, les $f$ côtés $a'_m$ qui forment la moitié gauche de la face supérieure de $\Pi$ sont choisis parmi les $\beta$ de la face inférieure, plus précisément dans un ordre (cyclique sur les indices) $\beta_1, \beta_2, \dots, \beta_k$, tandis que les $f$ côtés $a'_m$ qui forment la moitié droite de la face supérieure de $\Pi$ sont choisis parmi les $\alpha$ de la face inférieure, plus précisément dans un ordre (cyclique sur les indices) $\alpha_1, \alpha_2, \dots, \alpha_k$.

Nous vérifions que cette valuation respecte les trois contraintes énoncées\,:

\begin{enumerate}
\item Par définition, $v(p'_1) = v(p_1) = 0$. Par ailleurs, $v(p'_{2f+1}) = v(p'_1) + S \beta + S \alpha = 0 = v(p_{2f+1})$. Observons aussi au passage que la valeur du coin milieu de la face supérieure est $v(p'_{1 + f}) = v(p'_1) + S\beta = 0$.

\item En raison du \cref{lem-Skolem_copremier}, $\beta$ et $\alpha$ sont copremiers à $S$. Cela implique que les $S$ valeurs $v(p'_{1 + km})$ ($0 \le m \le S-1$) sont distinctes dans $\mathcal{S}(K_S)$ et donc qu'elles énumèrent tous les sommets de $K_S$ sans répétition, tout comme les $S$ valeurs $v(p'_{1 + km})$ ($S \le m \le 2S-1$) sont distinctes dans $\mathcal{S}(K_S)$ et qu'elles énumèrent donc aussi tous les sommets de $K_S$ sans répétition.  Par translation, il en va de même pour tous les ensembles $\{v(p'_{j + km})\}_{0 \le m \le S-1}$ et $\{v(p'_{j + km})\}_{S \le m \le 2S-1}$ ($1 \le j \le k$). Cela signifie que pour tout $1 \le j \le k$, tous les côtés $\beta_j$ (respectivement, tous les côtés $\alpha_j$) apparaissent une et une seule fois à gauche (respectivement, à droite) de la face supérieure de $\Pi$. Cela démontre l'existence de la bijection $\phi$.

\item Nous montrerons à l'Étape (4) que $\phi$ est suffisamment mélangeante. Pour l'instant, observons que par construction, $a'_1 = \overline{\beta_1(0)}$, de sorte que $v(a'_1) = \beta_1 \neq \alpha_1 = v(a_1)$. Similairement, $a'_{2f} = \overline{\alpha'_k(0)}$, de sorte que $v(a'_{2f}) = \alpha_k \neq  \beta_k = v(a_{2f})$. Finalement, alors que deux côtés consécutifs $(a_m, a_{m+1})$ sont d'un des trois types $(\alpha_j, \beta_j)$, $(\beta_j, \alpha_j)$ ou $(\beta_j, \alpha_{j+1})$ (avec $j \neq k$ dans cette dernière possibilité), deux côtés consécutifs $(a'_m, a'_{m+1})$ sont d'un des trois types $(\beta_j, \beta_{j+1})$ (où $k+1 := 1$), $(\alpha_j, \alpha_{j+1})$  (où $k+1 := 1$) ou $(\beta_k, \alpha_{1})$. Il est donc impossible d'avoir $\phi(m+1) = \phi(m) + 1$.
\end{enumerate}

\textbf{Étape 4}\,: Soient $v \in \mathcal{S}(K_S)$ et $q$ un coin de $\Pi$ tel que $v = v(q)$. Il nous reste à vérifier que tout autre coin $q'$ vérifiant $v = v(q')$ est identifié à $q$ dans le quotient. Pour ce faire, il suffit de vérifier qu'ensemble, les chaînes débutant par $q$ font intervenir tous les côtés $\alpha_j(v)$ et $\beta_j(v)$. \smallskip

\emph{Cas $v = 0$}. Sans perte de généralité, supposons que $q_{4j-3} := q$ débute $\overline{\beta_j(0)}$. (L'intérêt de débuter l'étiquettage de la chaîne de coins avec l'indice $4j-3$ se trouve dans un argument de récurrence.) Donc $q_{4j-3}$ est identifié au coin $q_{4j-2}$ qui débute $\beta_j(0)$\,; puisqu'il s'agit d'un côté $a_m$ avec $m$ pair, $q_{4j-2}$ termine le côté $\alpha'_j(0)$ (qui appartient aussi à $\Pi_j$). Donc $q_{4j-2}$ est identifié au coin $q_{4j-1}$ qui termine $\overline{\alpha'_j(0)}$.
\begin{itemize}
\item[$\bullet$] Si $1 \le j < k$, alors $q_{4j-1} \neq p'_{2f+1}$. Donc $q_{4j-1}$ débute le côté $\overline{\alpha_{j+1}(0)}$. Ainsi, $q_{4j-1}$ est identifié au coin $q_{4j}$ qui débute $\alpha_{j+1}(0)$, c'est-à-dire que $q_{4j}$ est le coin inférieur gauche $p_{2Sj + 1}$ de $\Pi_{j+1}$. Puisque $j+1 \ge 2$, $q_{4j}$ termine $\beta'_j(0)$\,; $q_{4j}$ s'identifie donc au coin $q_{4j+1} = q_{4(j+1)-3}$ qui termine $\overline{\beta'_j(0)}$ et qui débute le côté $\overline{\beta_{j+1}(0)}$. Une récurrence sur $1 \le j < k$ est alors clairement possible.

\item[$\bullet$] Si $j=k$, alors $q_{4k-1} = p'_{2f+1}$ est le coin supérieur droit de $\Pi$. Ainsi, $q_{4k-1}$ est identifié au coin inférieur droit de $\Pi$, $q_{4k} = p_{2f+1}$, qui termine $\beta'_k(0)$. En retour, $q_{4k}$ est identifié au coin $q_{4k+1}$ qui termine $\overline{\beta'_k(0)}$, c'est-à-dire au coin milieu $q_{4k+1} = p'_{f+1}$, qui débute $\overline{\alpha_1(0)}$. Donc $q_{4k+1}$ est identifié au coin $q_{4k+2}$ qui débute $\alpha_1(0)$, c'est-à-dire que $q_{4k+2} = p_1$ est le coin inférieur gauche de $\Pi$. Ce coin est ainsi identifié au coin $q_1 = p'_1$ qui débute $\overline{\beta_1(0)}$.
\end{itemize}
Par inspection, nous voyons que tous les $\alpha_j(0)$ et tous les $\beta_j(0)$ interviennent dans la chaîne précédente. Tous les coins de $\Pi$ appartenant à $v^{-1}(0)$ apparaissent donc dans la chaîne ci-dessus et ils sont ainsi tous identifiés dans le quotient.\smallskip

\emph{Cas $v \neq 0$}. C'est précisément ici que l'hypothèse «\,$k$ impair\,» sera utilisée. Soit $q_1 = q$ qui termine $\overline{\beta'_k(v)}$. Alors $q_1$ est identifié au coin $q_2$ qui termine $\beta'_k(v)$ et qui débute $\alpha_k(v)$ (puisque $v \neq 0$). Donc $q_2$ est identifié au coin $q_3$ qui débute $\overline{\alpha_k(v)}$. Encore du fait que $v \neq 0$, $q_3$ termine $\overline{\alpha'_{k-1}(v)}$ (où nous définissons l'indice $0 := k$) et est donc identifié au coin $q_4$ qui termine $\alpha'_{k-1}(v)$ et débute $\beta_{k-1}(v)$. Ainsi, $q_4$ est identifié au coin $q_5$ qui débute $\overline{\beta_{k-1}(v)}$.

Si $k=1$, alors $q_5$ débute $\overline{\beta_{1}(v)}$ et termine $\overline{\beta'_1(v)}$. Bref, $q_1 = q_5$. Dans ce cas, puisque $\alpha_1(v)$ et $\beta_1(v)$ ont été impliqués, tous les coins de valuation $v$ apparaissent dans la chaîne ci-dessus et sont ainsi tous identifiés dans le quotient.

Si $k = 2l+1$ avec $l \ge 1$, alors du fait que $v \neq 0$, $\overline{\beta_{k-1}(v)}$ est immédiatement précédé de $\overline{\beta'_{k-2}(v)}$.  Puisque $k-2 \not \equiv k \mbox{ mod } k$,  ce côté diffère de $\overline{\beta'_k(v)}$. Bref, $q_5 \neq q_1$ et $q_5$ est identifié au coin $q_6$ qui termine $\beta'_{k-2}(v)$. Puisque $v \neq 0$, $q_6$ débute $\alpha_{k-2}(v)$ et s'identifie au coin $q_7$ qui débute $\overline{\alpha_{k-2}(v)}$. Encore parce que $v \neq 0$, $q_7$ termine $\overline{\alpha'_{k-3}(v)}$ et s'identifie donc au coin $q_8$ qui termine $\alpha'_{k-3}(v)$ et qui débute $\beta_{k-3}(v)$. Donc $q_8$ est identifié au coin $q_9$ qui débute $\overline{\beta_{k-3}(v)}$.

En fait, par récurrence sur l'argument du dernier paragraphe, pour tout $1 \le m \le l$, le coin $q_5$ s'identifie au coin $q_{2 + 4m}$ qui débute $\alpha_{k-2m}(v)$, au coin $q_{3+4m}$ qui débute $\overline{\alpha_{k-2m}(v)}$, au coin $q_{4+4m}$ qui débute $\beta_{k-2m-1}(v)$ et au coin $q_{5+4m}$ qui débute $\overline{\beta_{k-2m-1}(v)}$. Puisque $k - 2m \not \equiv k \mbox{ mod } k$, tous ces coins et ces côtés sont distincts. Notons qu'à ce point-ci, tous les $\alpha_j(v)$ avec $j$ impair et tous les $\beta_j(v)$ avec $j$ pair sont intervenus.

Nous aboutissons alors au coin $q_{5 + 4l}$ qui débute  $\overline{\beta_{k}(v)}$ et qui termine $\overline{\beta'_{k-1}(v)}$. Par récurrence sur $0 \le m \le l-1$, nous voyons donc que ce coin est identifié au coin $q_{6+4(l+m)}$ qui termine $\beta'_{k-2m - 1}(v)$ et qui débute $\alpha_{k-2m-1}(v)$, qui lui s'identifie au coin $q_{7+4(l+m)}$ qui débute $\overline{\alpha_{k-2m-1}(v)}$ et qui termine $\overline{\alpha'_{k-2m-2}(v)}$, qui lui s'identifie au coin $q_{8+4(l+m)}$ qui termine $\alpha'_{k-2m-2}(v)$ et qui débute $\beta_{k-2m-2}(v)$, qui lui s'identifie au coin $q_{9+4(l+m)}$ qui débute $\overline{\beta_{k-2m-2}(v)}$.

Pour $m = l-1$, nous aboutissons au coin $q_{5 + 8l}$ qui débute $\overline{\beta_{1}(v)}$ et qui termine donc $\overline{\beta'_k(v)}$. Bref, $q_1 = q_{5+8l}$. Observons que tous les $\alpha_j(v)$ avec $j$ pair et tous les $\beta_j(v)$ avec $j$ impair sont apparus lors de cette deuxième récurrence. Donc tous les coins de valuation $v$ sont identifiés dans le quotient.



\section{\Cref{thm-asymptotique}} \label{sec-thm-asymptotique}

Nous allons montrer que pour tout genre $g \ge 1$, il existe un graphe $G_g \subset \Sigma_g$ cellulaire avec MC qui satisfait $V_r(G_g) \ge \sqrt{6g} + o(\sqrt{g})$. 
\\ \\
Tout d'abord, écrivons $\mathcal{G} := \{ g = g(S) := (S-1)(S-3)/6 \, | \, S \in \mathfrak{S} \}$ où  
\[  \mathfrak{S} := \left\{ S \in \mathbb{N} \, | \, S \mbox{ premier}, \, S \equiv 7 \mbox{ mod } 12 \right\} \, . \] Le premier élément de $\mathcal{G}$ est $g=4$. Dans ce qui suit, nous aurons besoin du contrôle suivant sur les écarts entre des éléments suffisamment grands de $\mathcal{G}$.
\begin{lem} \label{lem-gap}
Il existe $g_* \ge 7^{10}$ tel que pour tout $g' \in \mathcal{G}$ tel que $g' > g_*$, il existe $g'' \in \mathcal{G}_*$ tel que $g' < g'' < g' +  (g')^{9/10}$.
\end{lem}

\begin{proof}
Le fait suivant est un corollaire d'un théorème de Baker--Harman--Pintz \cite[Theorem 3(I)]{BHP} sur la répartition des nombres premiers appartenant à une classe de congruence\,:
\begin{quotation}
Soient $1 \le a < q$ deux entiers copremiers. Il existe un réel $x_*$ tel que pour tout $x \ge x_*$, il existe un nombre premier $p \equiv a \mbox{ mod } q$ entre $x$ et $x + x^{3/5}$. 
\end{quotation}
Pour $q = 12$ et $a = 7$, ceci implique qu'il existe un entier $S_*$ tel que pour tout entier $S' \ge S_*$, il existe $S'' \in \mathfrak{S}$ tel que $0 \le s := S'' - S' \le (S')^{3/5}$. Si $S' \in \mathfrak{S}$, considérons les éléments $g' := g(S')$ et $g'' := g(S'')$ dans $\mathcal{G}$ et posons $h := g'' - g'$. Nous estimons
\begin{align}
\notag h &= \frac{1}{6} \left[ (S''-2)^2 - (S'-2)^2  \right] = \frac{1}{6}(2(S'-2) + s)s  \le   \frac{1}{2}S' s \le \frac{1}{2} (S')^{8/5} \, .
\end{align}
Quitte à prendre $g_* \ge 7^{10}$, alors $S' = 2 + \sqrt{6g'+1} \le 3 + \sqrt{6g'}  \le \sqrt{24g'}$ et
\begin{align}
\notag h & \le \frac{1}{2} (24g')^{4/5} \le 7 (g')^{4/5} \le (g')^{9/10} \, .
\end{align}
\end{proof}
Fixons maintenant pour de bon un $g_*$ dont l'existence est attestée par le \cref{lem-gap} et dénotons par $\mathcal{G}_*$ l'ensemble des $g \in \mathcal{G}$ tels que $g > g_*$. 
\\ \\
Définissons les graphes $G_g$. Nous définissons d'abord $G_g$ pour $g \in \mathcal{G}$ en notant que le \cref{thm-optimalite}(b) a établi que $V_c(g) = 1 + \sqrt{6g+1} = \sqrt{6g} + o(\sqrt{g})$ pour tout $g \in \mathcal{G}$\, ; pour $g \in \mathcal{G}$, nous choisissons $G_g \subset \Sigma_g$ comme étant un graphe cellulaire avec MC qui réalise $V_c(g)$. Ensuite, nous définissons $G_g$ pour $1 \leq g \leq g_*$, $g \not \in \mathcal{G}$ comme suit : par la \cref{prop-existence} il existe une constante $0 < C_* <1$ telle que pour tout $1 \leq g \leq g_*$, il existe un graphe $H_g \subset \Sigma_g$ cellulaire avec MC tel que $V_r(H_g) > C_* \sqrt{6g}$, et donc nous prenons $G_g : = H_g$ pour $1 \leq g \leq g_*$, $g \not \in \mathcal{G}$.
\\ \\
Finalement, pour $g \not \in \mathcal{G}$ avec $g > g_*$, on définit $G_g$ récursivement par la construction de la somme connexe. C'est-à-dire, pour un tel $g$, nous laissons $g' \in \mathcal{G}$ être le plus grand élément dans $\mathcal{G}$ tel que $g' < g$ et posons $G_g := G_{g'} \# G_{h}$ où $h := g - g'$. Notons que le \cref{lem-gap} implique que $h < (g')^{9/10} < g'$ et donc cette définition récursive de $G_g$ pour $g > G_*$, $g \not \in \mathcal{G}$ est bien définie. Le c\oe ur de la démonstration consiste à montrer que ces graphes satisfont $V_r(G_g) \ge \sqrt{6g} + o(\sqrt{g})$.
\\ \\
Nous affirmons qu'il existe $D > 6$ tel que pour tout $k \ge 1$,
\[  S(G_k) \le D \sqrt{k}  \, . \]
Il convient de souligner que cette égalité tient si $k \in \mathcal{G}$, car dans ce cas (\cref{prop-graphecomplet}) $S(G_k) = 2 + \sqrt{6k+1} <  \sqrt{24 k}$. Il est aussi clair qu'une telle inégalité tient pour $k \le g_*$. Soit $g > g_* \ge 7^{10}$ avec $g \not \in \mathcal{G}$ et supposons que l'inégalité ait été établie pour tout $k < g$. Dénotons par $g'$ le plus grand élément de $\mathcal{G}$ qui soit inférieur à $g$ et $h := g - g'$. Posons aussi $G' = G_{g'}$, $H = G_h$ et $G = G_g = G \# H$. En vertu du \Cref{lem-gap}, nous avons $h < g^{9/10}$, de sorte que\,:
\begin{align}
\notag S(G) &\le S(G') + S(H) \le \sqrt{24 g'} + D \sqrt{h}  \\
\notag &\le \sqrt{24 g} + D g^{9/20} \\
\notag &\le \left( \dfrac{\sqrt{24}}{6} + 7^{- 1} \right) \, D \sqrt{g} < D \sqrt{g} \, .
\end{align}

Nous affirmons enfin qu'il existe $D' \ge D+1$ tel que pour tout $k \ge 1$,
\[  V_r(G_k) \ge \sqrt{6k} - D' k^{9/20} \, . \]
Cette inégalité tient pour $k \in \mathcal{G}$, car alors $V_r(G_k) = b(k) > \sqrt{6k}$ en vertu du \cref{thm-optimalite}(b). Une telle inégalité tient clairement pour $k \le g_*$. Soit $g > g_* \ge 7^{10}$ qui n'est pas dans $\mathcal{G}$ et supposons que nous ayons montré l'inégalité pour tout $k < g$. Avec les mêmes notations que précédemment, en ayant recours à l'identité ($\#$) concernant $V_r(G)$, nous calculons
\begin{align}
\notag V_r(G) &\ge V_r(G) \dfrac{S(G')}{S(G') + S(H)} + V_r(H) \dfrac{S(H)}{S(G') + S(H)} \ge  V_r(G') \dfrac{S(G')}{S(G') + S(H)} \\
\notag &\ge \sqrt{6g} - \left[ (\sqrt{6g} - \sqrt{6g'}) + (\sqrt{6g'} - V_r(G')) + V_r(G') \dfrac{S(H)}{S(G')}  \right] \, .
\end{align}
Nous avons déjà noté que $\sqrt{6g'} - V_r(G') < 0$. Puisque $\sqrt{1+x} - 1 \le x/2$ pour $x \ge 0$ et puisque $h \le (g')^{9/10}$, nous estimons
\begin{align}
\notag \sqrt{6g} - \sqrt{6g'} &\le \sqrt{6g'} \dfrac{h}{2g'} \le \sqrt{\frac{3}{2}} (g')^{2/5} \le  g^{9/20} \, .
\end{align}
Finalement, puisque $V_r(G') =  S(G') - 1$ (\cref{prop-graphecomplet}), nous calculons
\begin{align}
\notag  V_r(G') \dfrac{S(H)}{S(G')} &\le S(H) \le D\sqrt{h} \le D g^{9/20}  \, .
\end{align}
Il résulte de tout ceci que
\begin{align}
\notag V_r(G) &\ge \sqrt{6g} - (D+1) g^{9/20} \ge  \sqrt{6g} - D' g^{9/20} \, .
\end{align}


\appendix

\section{Invariants $B(g)$ et $C(g)$}\label{app_invariants}

       \begin{defn}   Soit $\Sigma_g$ une surface orientable fermée de genre $g$.  Définissons $B(g)$ comme le maximum d'éléments que peut contenir une famille de lacets non orientés plongés, tous disjoints deux à deux, et dont les classes d'homotopie libres (sans point de base) non orientées sont toutes différentes, et différentes de la classe triviale. 
       
         Définissons $C(g)$, pour un point de base $p \in \Sigma_g$ arbitraire,  comme le maximum d'éléments que peut contenir une famille de lacets non orientés plongés tous basés en $p$, disjoints deux à deux hors de $p$, et qui représentent des classes différentes en homotopie non orientée basée en $p$, et différentes de la classe triviale. 
         
         Définissons enfin $C_2(g)$, pour deux points de base arbitraires $p$ et $p'$ dans $\Sigma_g$, comme le maximum d'éléments que peut contenir une famille de courbes non orientées plongées reliant $p$ à $p'$, deux à deux disjointes, et qui représentent des classes différentes en homotopie non orientée à bouts $p$ et $p'$ fixés.  \end{defn}
            
            \begin{prop} $B(g) = 3g-3$ pour $g > 1$ et vaut $0$ en genre $0$ et $1$ en genre $1$.  Et $C(g) \leq 6g-3$ pour $g >1$,  vaut $0$ en genre $g=0$ et $3$ en genre $1$. Enfin $C_2(g) = C(g) + 1$. 
            \end{prop}
            
         Le résultat pour $B(g)$ est folklorique (nous en rappelons la preuve plus bas). Le fait que $C_2(g) = C(g) + 1$ est évident, il suffit de contracter n'importe quelle courbe entre $p$ et $p'$. 
               
\begin{proof}
Les cas particuliers en genre $0$ et $1$ sont évidents.                         
                       
             Commen\c{c}ons par $B(g)$ qui est un résultat classique : si une famille maximale est donnée, coupons $\Sigma_g$ le long des lacets, on obtient des surfaces $S_1, S_2, \ldots, S_k$.  Si la caractérisque d'Euler\footnote{Qu'il faudrait d'ailleurs appeler la caractéristique de Descartes-Euler, car Descartes fut le premier à énoncer la formule de la caractéristique en genre $0$ affirmant qu'elle était indépendante de la décomposition de la sphère en polygones. Euler en a fait la démonstration.} de $S_i$ est positive ou nulle, c'est un disque ou un anneau, ce qui contredit l'hypothèse d'indépendance des classes. Si elle est plus petite ou égale à $-2$, on peut trouver un lacet simple dans $S_i$ qui n'est pas librement homotope à son bord, ce qui contredit la maximalité. Si la caractéristique est $-1$, $S_i$ est soit un tore avec un trou, soit une paire de pantalons (une sphère avec trois trous). Le cas du tore troué est exclu car il contient un  lacet qui n'est pas homotope au bord. Donc chaque surface $S_i$ est une paire de pantalons et il y a exactement $2g-2$ telles paires sur une surface de genre $g$. Puisque chaque lacet est inclus dans deux des bords des $S_i$, il s'ensuit que $B(g) = (1/2)(3(2g-2)) = 3g-3$. 
            
              Abordons maintenant la preuve de l'inégalité $C(g) \leq 6g-3$ en genre supérieur à $0$. D'abord une définition :
              
              \begin{defn} Soit $\Sigma_g$ une surface orientable fermée de genre $g$. Soit $p \in \Sigma_g$ un point de base arbitraire. Soient $a$ et $b$ deux lacets plongés dans $\Sigma_g$, basés en $p$, disjoints hors du point $p$. On dit que la paire $(a,b)$ est une {\em paire duale de Poincaré}  si $[a] \cdot [b] = \pm 1$, lorsqu'on munit $a$ et $b$ d'orientations arbitraires, et où $[\cdot]$ est l'image par l'homomorphisme de Hurewicz de $\pi_1(\Sigma_g) \to H_1(\Sigma_g;\Z)$. 
              \end{defn} 
              
             Pour toute telle paire duale de Poincaré, on peut trouver un homéomorphisme de $\Sigma_g$ sur lui-même qui envoie la paire sur une paire standard $(a_i,b_i)$. 
             
             Supposons d'abord que la famille maximale $\mathcal F$ réalisant $C(g)$ soit entièrement formée de paires duales de Poincaré. On coupe la surface le long de ces lacets, et on obtient une sphère avec un trou, dont le bord est un $4g$-polygone $P$. C'est donc la fermeture convexe de $4g$ points dans le plan. Par identification des côtés de $P$ par paires duales, on a $2g$ lacets. Par le lemme précédent, on peut ajouter au maximum $4g-3$ arêtes. Donc, dans ce cas, on obtient pour $C(g)$ la valeur $2g + 4g-3 = 6g-3$. 
             
               A l'opposé, s'il n'y avait dans la famille maximale $\mathcal F$ aucune paire duale, on obtiendrait $3g-3$ lacets. En effet, la non-existence des paires duales de Poincaré signifie que, dans un petit disque près de $p$, les segments incidents à $p$ sont suivant l'ordre anti-horaire, de la forme (en notant $a_i$ et $a_i'$ les deux bouts d'un même lacet) :
                $$
               a_{i_1}, a_{i_2}, \ldots a_{i_{m-1}}, a_{i_{m}},  a_{i_m}', a_{i_{m-1}}', \ldots, a_{i_2}',  a_{i_1}'.
               $$
               C'est un ensemble emboité, et l'on peut donc retracter chaque lacet hors de $p$. On tombe alors dans la situation du calcul de B(g), et on obtient $3g-3$. Comme cette formule n'est pas valable pour le tore, il est préférable dans la suite de prendre $B(g) \leq 3g-2$ qui est valable pour tout $g > 0$. 
               
                 Voyons enfin les cas intermédiaires : on suppose qu'il existe sur la surface de genre $g$ exactement $g-g'$ paires duales dans une famille maximale pour $C(g)$. En coupant la surface le long de ces lacets, on obtient une surface $\Sigma$ de genre $g'$ avec un trou dont le bord est un $4(g-g')$-polygone $P$ avec les identifications usuelles sur les arêtes qui donnent déjà $2(g-g')$ lacets.  La formule exacte est alors $C(g) = 6g -3g' -3$ dans ce cas (quand $g'=0$, on retrouve la formule $6g-3$ et quand $g'=g$, on retrouve $3g-3$). Comme il suffit de démonter l'inégalité $C(g) \leq 6g -3g' -3$, on procède de la fa\c{c}on suivante. Aux $2(g-g')$ lacets, on ajoute les $4(g-g') -1$ lacets locaux autour de $P$ qui relient un sommet $p$ de $P$ aux autres. Il reste alors les lacets de $\Sigma$ basés en $p$. Mais comme il n'existe plus de paire Poincaré duale, ces lacets se rétractent dans $\Sigma$ hors de $p$ et constituent une famille maximale pour $\Sigma_{g'}$ dont le nombre est $B(g) \leq 3g-2$. La somme donne $C(g) \leq 2(g-g') + 4(g-g') - 1 + 3g' -2 = 6g - 3g' - 3$.  Le maximum de cette borne supérieure sur $0 \leq g' \leq g$ est $6g-3$.
\end{proof}

%
%

\end{document}